\documentclass[12pt, a4paper, twoside, notitlepage]{article}
\usepackage{amssymb}
\usepackage{amsthm}
\usepackage[centertags]{amsmath}
\usepackage{accents}
\usepackage{todonotes}
\usepackage{enumerate}
\usepackage{bbm}
\usepackage[english]{babel}
\usepackage{multimedia}
\usepackage[latin1]{inputenc}
\usepackage{lineno}
\usepackage{times}
\usepackage[T1]{fontenc}
\usepackage{picinpar}    
\usepackage{color}
\usepackage{graphicx}
\usepackage{graphics}
\usepackage{pgf}
\usepackage{marginnote}
\usepackage{todonotes}
\usepackage{sidenotes}
\newcommand{\SLsmall}{{\bf L}^{\circ}}
\newcommand{\SLxsmall}{{L}_x^{\circ}}
\newcommand{\SLysmall}{{L}_y^{\circ}}
\newcommand{\SL}{{\bf L}}
\newcommand{\SLx}{{L}_x}
\newcommand{\SLy}{{L}_y}

\newcommand{\KISO}{\mathbb{K}_{iso}}
\newcommand{\KISOSMALL}{\mathbb{K}_{iso}^{small}}
\newcommand{\KGM}{\mathbb{K}_{GM}}
\newcommand{\KGMsmall}{\mathbb{K}_{GM}^{small}}

\newcommand{\vbolus}{{ v}^*}

\newcommand{\wbolus}{w^*}

\newcommand{\GMOP}{\mathfrak{D}_{\KGM}}
\newcommand{\REDIOP}{\mathfrak{D}_{\KISO}}

\newcommand{\REDIOPSMALL}{\mathfrak{D}_{\KISOSMALL}}
\newcommand{\GMOPSMALL}{\mathfrak{D}_{\KGMsmall}}

\newcommand{\rhoreg}{{\tilde{\rho}}}
\newcommand{\rhobar}{{\tilde{\rho}}}
\newcommand{\VV}{{V_{v}}}
\newcommand{\VT}{{V_{\theta}}}
\newcommand{\VS}{{V_{S}}}
\newcommand{\VC}{{V_{C}}}
\newcommand{\HV}{{H_{v}}}
\newcommand{\HT}{{H_{\theta}}}
\newcommand{\HS}{{H_{S}}}

\newcommand{\HtwoV}{{H_{v}^{2}}}
\newcommand{\HtwoT}{{H_{\theta}^{2}}}
\newcommand{\HtwoS}{{H_{S}^{2}}}

\newcommand{\tcons}{\theta}

\newcommand{\salabs}{S}
\newcommand{\salgen}{{S}}
\newcommand{\Omegaint}{{{\Omega}_{0,1}}}
\newcommand{\Omegainthalf}{{{\Omega}_{\scriptstyle 1,2}}}
\newcommand{\Lsq}{{L^{2}(\Omega)}}
\newcommand{\Hone}{{H^{1}(\Omega)}}
\newcommand{\Htwo}{{H^{2}(\Omega)}}

\newcommand{\Htwoint}{{H^{2}(\Omegaint)}}
\newcommand{\Honeinthalf}{{H^{1}(\Omegainthalf)}}
\newcommand{\Htwointhalf}{{H^{2}(\Omegainthalf)}}
\newcommand{\taperbound}{\pi_{\partial\Omega}}
\newcommand{\LinftyT}{L^\infty_{I_\theta}(\Omega)}

\newcommand{\LinftyS}{L^\infty_{I_S}(\Omega)}

\newcommand{\LinftyC}{L^\infty_{I_C}(\Omega)}

\newcommand{\Linftyrho}{L^\infty_{I_\rho}(\Omega)}
\newcommand{\LinftyTS}{L^\infty_{{I_\theta}\times I_{S}}(\Omega)}
\newtheorem{theorem}{Theorem}[section]
\newtheorem{lemma}[theorem]{Lemma}

\newtheorem{corollary}[theorem]{Corollary}
\newtheorem{definition}[theorem]{Definition}
\newtheorem{remark}[theorem]{Remark}
\begin{document}
\begin{titlepage}
\title{Global Well-Posedness of the Primitive Equations of Large-Scale Ocean Dynamics with  the Gent-McWilliams-Redi Eddy Parametrization Model}
\author{Peter Korn\footnote{\small Max Planck Institute for Meteorology, Hamburg, Germany}, Edriss S. Titi
\footnote{
Department of Mathematics, Texas A \& M University, College Station, TX 77843, USA, 
Department of Applied Mathematics and Theoretical Physics, University of Cambridge CB3 0WA, UK and
Department of Computer Science and Applied Mathematics, Weizmann Institute of Science, Rehovot 76100, Israel,
}
 }

\maketitle
\setcounter{page}{1}

\begin{abstract}
\it{
We prove global well-posedness of the ocean primitive equations coupled to advection-diffusion equations of the oceanic tracers temperature and salinity that are supplemented by the eddy parametrization model due to Gent-McWilliams and Redi.  This parametrization forms a milestone in global ocean modelling and constitutes a central part of any general ocean circulation model computation. The eddy parametrization adds a secondary transport velocity to the tracer equation and renders the original Laplacian operators in the advection-diffusion equations nonlinear, with a diffusion matrix that depends via the equation of state in a nonlinear fashion on both tracers simultaneously. The eddy parametrization of Gent-McWilliams-Redi augments the complexity of the mathematical analysis of the whole system which we present here. We show first that weak solutions exist globally in time, 
provided the parametrization uses a regularized density. Then we prove by a detailed analysis of the eddy operators the global well-posedness. 
Our results apply also to the ``small-density slope approximation'' that is commonly used in global ocean simulations.   
}
\end{abstract}
\end{titlepage}

%
\section{Introduction}
The hydrostatic Boussinesq equations of the ocean, also referred to as the {\it ocean primitive equations}, are the classical dynamical equations used in ocean and climate research (see e.g. \cite{IPCC, CORE}). A central component is the subgridscale parametrization of mesoscale eddies due to Gent-McWilliams and Redi \cite{GENT_MCWILLIAMS}, \cite{REDI}.  This parametrization is a major achievement in the science of ocean modelling \cite{GM_Science} and forms a cornerstone of global ocean modelling. Virtually every global ocean general circulation model relies on, and incorporates, this parametrization (see e.g. \cite{CORE}). The mathematical analysis of the ocean primitive equations with the Gent-McWilliams-Redi parametrization forms the topic of this paper. 

The ocean primitive equations model the large-scale circulation of the ocean as a thin layer of a weakly compressible fluid on a rotating sphere under the Boussinesq  and the hydrostatic approximation. 
The ocean primitive equations consist of a coupled system of evolution equations for horizontal velocity $v=(v_1,v_2)$, potential temperature $\theta$ and salinity $S$ that is closed by an equation of state 
 that describes the density as function of temperature, salinity and static pressure. 
In this work we investigate the ocean primitive equations with Gent-McWilliams-Redi eddy parametrization, which in cartesian coordinates 
are given by
\begin{subequations}\label{PE_OCEAN}
\begin{align}
&\partial_t v+(v\cdot\nabla) v +w\partial_z v 
+\nabla p 
+f\vec{k}\times v 
-\frac{1}{Re_1}\triangle\,v - \frac{1}{Re_2}\partial^2_{zz} v=0,\label{PE_OCEAN_1}\\
&\partial_z p +g\rho=0, \label{PE_OCEAN_2}\\
&\nabla\cdot v+\partial_zw=0,\label{PE_OCEAN_3}\\ 
&\partial_t \theta+ (v\cdot\nabla) \theta 
+w\partial_z \theta
+\REDIOP(\rhobar, K_I, K_D)(\theta)+\GMOP(\rhobar, \kappa)(\theta)
=0,\label{PE_OCEAN_4}\\
&\partial_t S+ (v\cdot\nabla) S 
+w\partial_z S
+\REDIOP(\rhobar, K_I, K_D) (S)+\GMOP(\rhobar, \kappa)(S)=0,\label{PE_OCEAN_5}
\end{align}
\end{subequations}
with the equation of state $\rho=\rho(\theta,S,p_{st})$,
where $p_{st}$ denotes the static pressure (see (\ref{GIBBS_PRESSURE_BOUSSINESQ}) ), $w$ the vertical velocity, $\rho$ the equation of state, $f$ the Coriolis parameter, and $g$ the gravitational constant.
The numbers $Re_1,Re_2$ denote horizontal and vertical viscosity, $K_I,K_D$ are the isoneutral and dianeutral diffusivity, $\kappa$ is referred to as eddy advection parameter. The operators $\nabla, \nabla\cdot$ and $\triangle$ denote the horizontal gradient, divergence and Laplacian, respectively. 

The above system (\ref{PE_OCEAN}) differs from the classical primitive equations by introducing in the equations for the oceans active tracers temperature and salinity  (\ref{PE_OCEAN_4}), (\ref{PE_OCEAN_5}), the {\it isoneutral eddy diffusion operator} $\REDIOP$ and the {\it eddy advection operator} $\GMOP$. 
The presence of these two nonlinear, anisotropic and flow-dependent operators changes the nature of the originally linear advection-diffusion equations (\ref{PE_OCEAN_4}), (\ref{PE_OCEAN_5}). Therefore in comparison to the classical analysis of the primitive equations  the additional nonlinearity pose the main challenge in the mathematical analysis of this paper.

The work presented here capitalizes on the large body of literature of the mathematical analysis of the classical primitive equations,  pioneered  by Lions-Temam-Wang \cite{LIONS_TEMAM_WANG_2}, where the global existence of {\it weak solutions} was proven. Cao-Titi established in \cite{CaoTiti} global existence and uniqueness of {\it strong solutions} for initial data in the Sobolev space $H^1$
and with Neumann boundary conditions for velocity on top and bottom of the domain. A different proof is given by Kobelkov \cite{Kobelkov}. For Dirichlet boundary conditions well-posedness in $H^1$ was obtained by Kukavica -Ziane \cite{KUKAVICA_ZIANE1}. These works do rely on energy estimates. 
A  $L^p$-approach for the primitive equation was developed in Hieber-Kashiwabara \cite{HIEBER}.
The extension of global well-posedness to nonlinear equations of state can be found in Korn \cite{KORN_NLEOS}. 

The ocean modelling literature is flexible about the exact specifications in the mathematical definitions of the operators $\REDIOP, \GMOP$. The original publication by Gent-McWilliams \cite{GENT_MCWILLIAMS} (see also \cite{GENT_WILLEBRAND}) formulates on one hand the eddy closure in the form of a (continuous) PDE, but on the other hand develops its physical arguments that results in the specific form of the eddy closure in terms of (finite dimensional) numerical models. 
This ambiguity suggests an interpretation of the continuous ocean model equations in an averaged sense, where the averages are related to specific grid sizes (cf. chapter 8 in \cite{GRIFFIES_BOOK}). Our work investigates the mathematical structure of this important ocean parametrization 
from the perspective of mathematical analysis of hydrodynamic PDE. We investigate the eddy parametrization of Gent-McWilliams-Redi in the continuous limit
as a nonlinear PDE. 

We start with establishing the existence of weak solutions for (\ref{PE_OCEAN}). 
For weak solutions a density regularization inside the eddy parametrization is necessitated to prove a' priori bounds on the advective part (the ``GM-part'')
of the eddy model and for showing the convergence to a weak solution.
For the special case of the {\it small density slope approximation} -a simplification of the full eddy operator, preferably used in ocean models, under the physically appropriate assumption that density slopes are small-  the a' priori estimates for the finite dimensional approximative system can be obtained without regularization. This explains the stability of numerical codes, but still allows for noisy solutions. For proving the convergence of the approximative to weak solutions we have to regularize as in the full tensor case. The uniqueness of weak solutions is open.   

For proving existence and uniqueness of (strong) solutions to  (\ref{PE_OCEAN}) globally in time within the Sobolev space $H^1$ we apply energy estimates, following the approach of Cao-Titi \cite{CaoTiti}. 
The regularity of strong solutions demands a' priori estimates in $H^1$ by integrating against the Laplacian of test functions.
In this work we replace in the temperature and salinity equations the Laplace operator by  $\REDIOP$. This second-order diffusion operator has divergence-gradient structure, its mixing tensor 
is anisotropic, nonlinear and time-dependent. Establishing the a' priori $H^1$-estimates requires control of the spatial and temporal regularity of the elements of 
the symmetric mixing tensor of $\REDIOP$ and of the skew-symmetric tensor of $\GMOP$, which both are constructed via functions of density derivatives. 
This additional complexity, absent in the previous works cited above, requires us to prove the following novel results:
\begin{enumerate}
\item We prove  in Lemma \ref{LEMMA_ELLIPTIC_REG_REDI} an elliptic regularity result for the nonlinear diffusion operator $\REDIOP$. 
This is used in the $L^\infty_tH^1_x$-bound on temperature and salinity. Elliptic regularity is classical for the Laplace operator but novel for $\REDIOP$. 
\item We estimate the temporal regularity of the mixing tensor by means of an evolution equation for density. Such an equation is not part of the original system (\ref{PE_OCEAN}) of primitive equations but it can be derived from the equation of state. From this equation we estimate the rate of change of the gradient of density in time (see (\ref{density_eq}), Lemma \ref{LEMMA_DIFFERENCES} and (\ref{PROOF_STRONG_SOLUTION_2_time_deriv4a})).
\item We prove a  $L^\infty_tH^1_x$ bound for the eddy induced advection velocity, given by by $\GMOP$
(see (\ref{PROOF_STRONG_SOLUTION_3})- (\ref{PROOF_STRONG_SOLUTION_3vert})).
\end{enumerate}
Each of these three results is crucial to our proof strategy 
and for all three results we need to regularize the density within the eddy operator. The reason being the nonlinearity of the equation of state and the fact that density enters the eddy operators $\REDIOP, \GMOP$ in the form of derivatives. For more details we refer to Section \ref{SECT_Background} and to Remarks \ref{remark_three_derivatives} and \ref{remark_minimal_regularization}.

The original (unregularized) eddy closure model of Gent-McWilliams-Redi is analytically out of reach due to the highly nonlinear and flow-dependent 
nature of the eddy closure that lead to potential singularities and to degeneracies near the boundaries. The computational ocean modelling community has responded to this challenge with adhoc regularizations to modify the behaviour near the boundary and to compute sufficiently smooth solution that agree with physical expectations (see e.g. \cite{Lemarie, MOM_DOCU}). The discretization of the eddy operators $\REDIOP, \GMOP$ is a subtle numerical problem, for details we refer to   
\cite{GRIFFIES_1, GRIFFIES_2, KORN_PARAM}. Guided by these considerations we introduce rigorously a physically and computationally motivated regularization that we then study analytically. This provides in turn a rigorous basis for computational practice in ocean modelling. 


{\bf Organization of the paper.}
Section \ref{SECT_Background} provides physical background on mesoscale eddies and the need for regularization. Section \ref{SECT_THERMPDYN} defines the thermodynamcial concept. In section \ref{SECT_PARAMETRIZATIONS} we introduces the regularized eddy operators.
Section \ref{SECT_STRONG_SOLUTION} contains the main result, of this paper.  The proofs are given in Sections \ref{SECT_PROOF_MAIN_RESULT} and \ref{SUBSECT_PROOF_EX_STRONG}.
\section{Physical and Modelling Background}\label{SECT_Background}
Mesoscale eddies have a typical horizontal size between 10-100 km and a lifetime between 10-100 days. They belong to the most energetic phenomena in the circulation of the world's ocean and are also referred to as ``weather" of the ocean due to their analogy with atmospheric weather systems. 
They are predominantly created by baroclinic instabilities on a scale that is determined by the Rossby deformation radius. Ocean eddies are much smaller than their atmospheric counterparts and global ocean general circulation models have to invest significant  computational resources to resolve parts of the mesoscale eddy spectrum. 

The purpose of the eddy-induced diffusion is to mimic the behaviour of eddies to enhance mixing locally preferably along directions in which a water parcel can move in an adiabatic way without changing buoyancy (see e.g. \cite{Fox-Kemper_Bryan} for a review). The diffusion along the isoneutral direction is by orders of magnitude larger than the diffusion perpendicular to this directions \cite{Ledwell}. The parametrization of eddy-induced advection by Gent-McWilliams-Redi aims to capture the effect of mesoscale eddies on tracers by means of stirring through an adiabatic and potential energy reducing closure 
(\cite{GENT_MCWILLIAMS}, \cite{GENT_WILLEBRAND}, \cite{VALLIS}). 
Eddies convert potential into kinetic energy, thereby flattening of neutral surfaces reduces potential energy, while preserving the mass between isoneutral layers. 

The eddy advection, originally formulated by Gent-McWilliams for coarse resolution ocean models in vertical density coordinates \cite{GENT_MCWILLIAMS}, was transformed to more frequently used depth-coordinates in \cite{GENT_WILLEBRAND} and combined with the eddy-induced diffusion of Redi \cite{REDI}. A common modelling framework of eddy-induced advection and diffusion was developed in \cite{GRIFFIES_1, GRIFFIES_2}, in this form it was used in numerous ocean modelling studies (see e.g. \cite{CORE, CORE_2}. The Gent-McWilliams-Redi parametrization removes essential model biases of coarse resolution models such as the Veronis effect (see, e.g. \cite{GENT_2020}). Among its effects is a poleward heat transport in better agreement with observations or a more realistic occurrence of convective regions (see, e.g., \cite{DanabasMcWilliams}, \cite{GENT_2020}).

The eddy parametrization approximates 
subgrid scale fluxes of density, 
it is essential to use therefore an accurate equation of state. 
Unlike the case of atmospheric dynamics no sufficiently accurate  analytical expression of the equation of state 
is available for oceanic dynamics.
We use the equation of state TEOS-10 \cite{EOS10} of the Intergovernmental Oceanographic Commission, the official description of seawater properties in marine science. It provides an accurate description of seawater thermodynamics by blending the Gibbs formalism of thermodynamics with seawater measurements  (see, e.g. \cite{Feistel2003}). This implies the natural condition that tracers are bounded within an admissible range 
given by the equation of state. 

{\bf Parametrizations and Need of Regularization.} 
The purpose of the eddy parametrization 
is to capture the mean effect of subgrid scale buoyancy fluxes onto the resolved large-scale circulation. 
This particular parametrization, as well as others,  expresses unresolved quantities in terms of resolved through contructing nonlinear functions of partial derivatives
of resolved quantities. The formulation of such a nonlinear function is not derived from first principle but guided by physical ad-hoc consideration, 
applied to a very specific individual physical process. In our case this is the process of baroclinic instability. The mathematical object constructed thereby is then supplemented by ``numerical controllers'' such as limiters and thresholds in order to make the eddy parametrization computable in a general model of ocean dynamics   such as the ocean primitive equations with all its complexities. The fundamental approach to approximate subgrid scale processes of density by nonlinear functions of partial density derivatives challenges the available regularity of strong solution. In essence, this accumulation of derivatives generates a nonlinearity that is stronger than the quadratic nonlinearity of the fluid dynamics part.

The following formal argument illustrates the necessity of a density regularization for the eddy parametrization. The typical building block of the  eddy operator is of the form $\partial_{x_i}(k_{ij}\partial_{x_i}\theta)$, where $k_{i,j}$ denotes a matrix element of the eddy operator.  For a strong solution $(v,\theta,S)$ 
the eddy operator satisfies $\partial_{x_i}(k_{ij}\partial_{x_j}\theta)\in L^2_t L^2_x$, or, equivalently $k_{ij}\partial_{x_i}\theta\in L^2_t H^1_x$.  We have 
$k_{ij}\sim \partial_{x_i}\rho$ ($\sim$ indicates the same regularity class) and $\partial_{x_i}\rho\sim \partial_{x_i}\theta$. Consequently, having  
$k_{ij}\partial_{x_i}\theta\in L^2_t H^1_x$ requires that the product satisfies $\partial_{x_i}\theta\partial_{x_j}\theta\in L^2_t H^1_x$ - a property  
that is beyond the regularity of strong solutions. We do not know if this is an artefact of our proof strategy or a generic phenomenon related to the parametrization. The mathematical challenges are described in Remarks \ref{remark_three_derivatives} and \ref{remark_minimal_regularization}. The computational experience of ocean circulation models supports the need of a regularization to provide noise-free solutions (see e.g. \cite{Lemarie, MOM_DOCU}). 

A physical viewpoint suggests that the energy extraction and transfer from potential to kinetic energy, approximated by the eddy parametrization, is related to a stable and smooth background stratification. The regularity of such a background stratification can be described mathematically by filtering the instantaneous density field. Observations in the world ocean of the buoyancy frequency $N:=(-g\rho_0\partial_z\rho)^{1/2}$, which determines the denominator of the density slopes (cf. \ref{SLOPE_DEF_1}) reveal a field  that is ``extremely noisy" with ``high wave-number jitter'' to be removed to make it interpretable (citations from \cite{Wunsch}, see their Fig.2 ). It appears a physical fact that the buoyancy frequency $N$ and the density slopes are non-smooth physical quantities. The design of the Gent-McWilliams-Redi eddy closure makes an implicit smoothness assumption, that was justified practically by the very coarse resolution of the numerical 
ocean models to which it was applied initially, where the coarse mesh acts as a low-pass filter. In the PDE limit this filter disappears, the implicit assumption becomes apparent and we restore regularity via a regularization.     

 \section{Ocean Thermodynamics}\label{SECT_THERMPDYN}
The oceanic equation of state TEOS-10  \cite{EOS10, NEMO} 
assumes that temperature, salinity and pressure values are {\it within} the observed range of the worlds ocean. 
Accordingly we make the following (cf. \cite{KORN_NLEOS})\\
\noindent {\bf Hypothesis on Ocean Thermodynamics: } 
\begin{enumerate}
\item There exist intervals of {\it physically admissible ranges} for temperature $I_\theta:=[-K_1, K_2]$,  salinity $ I_\salgen:=[L_1,L_2]$, 
and pressure $I_p:=[M_1,M_2]$ with $K_1,K_2$, $L_1,L_2,M_1,M_2>0$, such that for
$(x,y,z,t)\in\Omega\times [0,T]$ 
\begin{equation}\begin{split}\label{GIBBS_RHO_ASSUMPTION1}
 \theta(x,y,z,t)\in I_\theta,\quad S(x,y,z,t)\in I_S,\quad p(x,y,z,t)\in I_p. 
\end{split}\end{equation}
\item  If $\theta\in I_\theta, \salgen\in I_\salgen$, $p\in I_p$, then the coefficients $c_{i,j,k}$ of the equation of state, defined below in (\ref{GIBBS_SPECIFIC_rho}), are determined such that the density $\rho$ stays within the admissible range
\begin{equation}\begin{split}\label{GIBBS_RHO_ASSUMPTION2}
 \rho\in I_\rho:=[R_1,R_2], \quad R_1,R_2>0.
\end{split}\end{equation}
\end{enumerate}
For the Boussinesq approximation one replaces the density $\rho$ by a reference density $\rho_0>0 $ in 
all terms of the dynamical equations except for the buoyancy term in the momentum balance and in the equation of state (see e.g. Section 2.4. in \cite{VALLIS}).
For energetic consistency one uses in the equation of state the {\it static pressure} 
\begin{equation}\begin{split}\label{GIBBS_PRESSURE_BOUSSINESQ}
 p_{st}(z):=-g\rho_0z,\quad\text{for }z\leq 0.
\end{split}\end{equation}
\begin{definition}\label{THERMODYNAMIC_DEF}
Let the hypothesis on ocean thermodynamics (\ref{GIBBS_RHO_ASSUMPTION1}) and (\ref{GIBBS_RHO_ASSUMPTION2}) hold.\\
$i)$ Define $\LinftyT:=\{\theta\in L^\infty(\Omega): \theta\in I_\theta\}$, $\LinftyS:=\{S\in L^\infty(\Omega): S\in I_S\}$.\\
$ii)$ Let $\theta\in \LinftyT, S\in \LinftyS$. The density $\rho$ 
is given via the specific volume $\mathrm{v}$
\begin{equation}\begin{split}\label{GIBBS_SPECIFIC_rho}
&\rho(\tcons,\salabs,p_{st}):=\mathrm{v}(\tcons,\salabs,p_{st})^{-1},\quad
\text{with }\mathrm{v}(\tcons,\salabs,p_{st}):=\sum_{i,j,k=0}^{N_i,N_j,N_k}c_{i,j,k}\salabs^i \tcons^j p_{st}^k,
\end{split}\end{equation}
with given coefficients $c_{i,j,k}\in \mathbb{R}$ (cf table K1 in \cite{EOS10}), $N_i=7,N_j=6,N_k=5$.
\end{definition}
\begin{lemma}[Boundedness of Potential Temperature, Salinity and Density]\label{LEMMA_TRACER_BOUNDED}
Denote by $C\in\{\theta,S\}$ the solution of the tracer equation for temperature (\ref{PE_OCEAN_4}) or 
for salinity (\ref{PE_OCEAN_5}).\\
$i)$ Let initial conditions satisfy $C_0\in \VC\cap \LinftyC$. Then it holds 
\begin{equation}\label{LEMMA_TRACER_BOUNDED1}
  C(t)\in \LinftyC.
\end{equation}
$ii)$
Let  $\theta_0\in\VT\cap\LinftyT, S_0\in\VS\cap\LinftyS$ be initial conditions of the tracer equation for temperature (\ref{PE_OCEAN_4}) and for 
for salinity (\ref{PE_OCEAN_5}). 
Then it holds 
\begin{equation}\label{PROPOSITION_BOUNDED_DENSITY_1}
\rho(t)\in \Linftyrho.
\end{equation}
\end{lemma}
\begin{proof}
The proof under the conditions of Definition \ref{THERMODYNAMIC_DEF} and without eddy parametrization can be found in \cite{KORN_NLEOS}. It is extended to the case under consideration here by Lemma \ref{LEMMA_PROP_REDIOPERATOR} and the skew-symmetry of $\KGM$.
\end{proof}

\begin{lemma}[\cite{KORN_NLEOS}]\label{LEMMA_DENSITY_DIFF}
Let $(\tcons_1,\salabs_1)\in \LinftyTS$ and $(\tcons_2,\salabs_2)\in \LinftyTS$ be two pairs of temperature and salinity fields,  
denote by 
 $\rho_1(\tcons_1,\salabs_1)$, $\rho_2(\tcons_2,\salabs_2)$ the two associated densities. 
Then it holds almost everywhere in $\Omega$ 
\begin{equation}\begin{split}\label{LEMMA_DENSITY_DIFF1}
&|\rho_1-\rho_2|\leq K(|\tcons_1-\tcons_2|+|\salabs_1-\salabs_2|),
\end{split}\end{equation}
where $K=K(\theta_1,\theta_2,S_1,S_2,c_{i,j,k})$ is bounded. 
\end{lemma}
\noindent {\bf Density Equation.} 
For the density an evolution equation can be derived by differentiating equation (\ref{GIBBS_SPECIFIC_rho}) with respect to time. This yields
\begin{equation}\begin{split}\label{density_eq}
\partial_t\rho&=-a\partial_t\theta +b\partial_t S=aF_\theta +b F_S\\
\text{with }a&:=-\frac{\partial\rho}{\partial\theta},\quad b:=\frac{\partial\rho}{\partial S}, \\
F_\theta&:=- (v\cdot\nabla) \theta 
-w\partial_z \theta
-\REDIOP(\rhobar)(\theta)-\GMOP(\rhobar)(\theta),\\
F_S&:=
- (v\cdot\nabla) S
-w\partial_z S
-\REDIOP(\rhobar)(S)-\GMOP(\rhobar)(S).
\end{split}\end{equation}
\section{Ocean Eddy Parametrization}\label{SECT_PARAMETRIZATIONS}
We introduce a mathematical formalization of oceanic neutral physics in section \ref{SECTION_NEUTRAL_DENSITY}.
This forms the basis of the definition of the eddy parametrizations in \ref{SUBSECTION_ISONEUTRAL_DIFFUSION}. 
\subsection{Regularized Isoneutral Density Slopes}\label{SECTION_NEUTRAL_DENSITY}
\vskip-0.15cm
\begin{figure}[ht]
\begin{center}
\includegraphics*[width=0.42\textwidth]{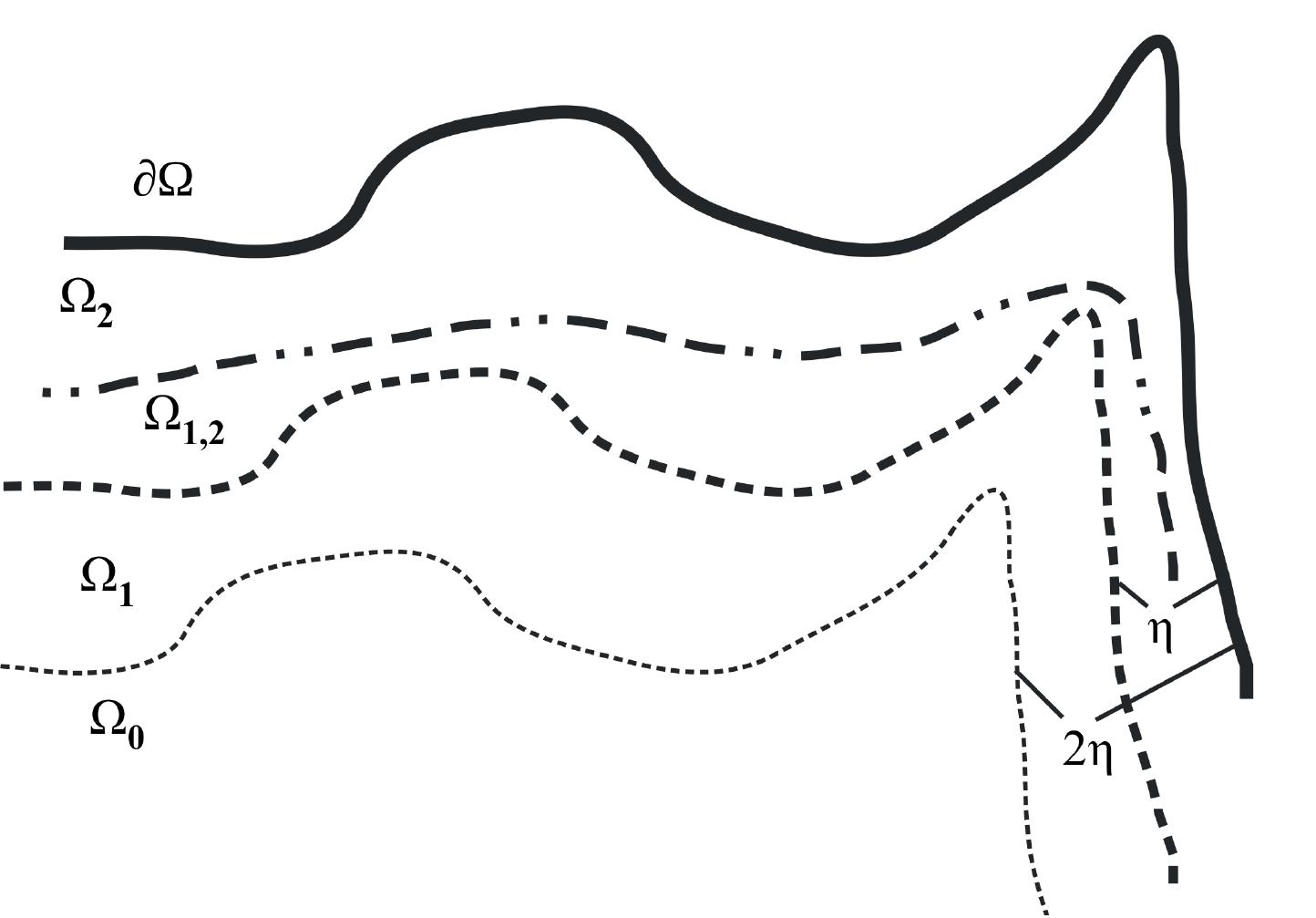}
\caption{\it\small Illustration of the domain near the boundary. In $\Omega_0$, at a distance of at least $2\eta$ away from the boundary the isoneutral density slope $\SL$ is not modified by the cut-off function $\taperbound$. At a distance less than $2\eta$ and more than 
$\eta$, in $\Omega_1$,  the density lope $\SL$ is reduced by the cut-off function $\taperbound$, at a distance less than 
$\eta$, in $\Omega_2$ both vanish due to the cut-off function. 
}
  \label{fig:SPIN_UP}
\end{center}
\end{figure}
 The eddy parametrization acts in the interior of the ocean and away from the boundary and only if a minimal stratification exists. For this purpose we decompose the domain according to the distance to the boundary (see Fig. \ref{fig:SPIN_UP}) 
\begin{equation}\begin{split}\label{DOMAIN_DECOMP}
\Omega=\Omega_0\cup\Omega_1\cup\Omega_2,
\end{split}\end{equation}
with 
$\Omega_0:=\{(x,y,z)\in\Omega:\, dist((x,y,z),\partial\Omega) \geq 2\eta\}$, 
$\Omega_1:=\{(x,y,z)\in\Omega:\, \eta <dist((x,y,z),\partial\Omega) < 2\eta\}$
and $\Omega_2:=\{(x,y,z)\in\Omega:\, dist((x,y,z),\partial\Omega) \leq \eta\}$.
Furthermore we define the sets
\begin{equation}\begin{split}\label{INTERIOR_DEF}
\Omegaint:=&\Omega_0\cup\Omega_1=\{(x,y,z)\in\Omega:\, dist((x,y,z),\partial\Omega) > \eta\},\\ 
\text{and }\ \Omegainthalf:=&\{ (x,y,z)\in\Omega:\, dist((x,y,z),\partial\Omega)> \frac{\eta}{2}\},
\end{split}\end{equation}
in such a way that the boundary $\partial \Omegainthalf$ is of class $C^2$. 
We use the ``cut-off'' function'' $\pi_{\partial\Omega}$ to define isoneutral density slopes in the interior ocean.
Let $\xi:[0,\infty]\to [0,1]$ be a $C^\infty$ monotonic nondecreasing function 
\begin{equation}\begin{split}
\xi(s)=
\begin{cases}
&0,\quad\text{if}\quad 0\leq s < 1,\\
&1,\quad\text{if}\quad s\geq 2.
\end{cases}
\end{split}\end{equation}
Define for a given $\eta >0$ and $(x,y,z)\in\bar{\Omega}$ the cut-off function 
\begin{equation}\begin{split}\label{Cut-Off}
\taperbound(x,y,z):=\xi(\frac{dist((x,y,z),\partial\Omega)}{\eta}).
\end{split}\end{equation} 
The physical requirement of a minimal stratification is implemented by 
the tapering function $\Pi_{s_0,\epsilon_0}$ in Definition \ref{REG_NEUTRAL}.

\begin{definition}[Regularized Neutral Physics]\label{REG_NEUTRAL}
Suppose $\theta\in \Hone\cap\LinftyT, S\in \Hone\cap\LinftyS$.\\
$i)$ Let $j\in C^\infty(\Omega)$ be a regularizing kernel with $j > 0$ on $\Omega$, $\int_\Omega j\, dx=1$, 
$j_\eta(x):=\eta^{-3} j(\frac{2x}{\eta})$. Define for  density $\rho\in \Linftyrho$ the {\bf regularized density $\rhoreg$} by
\begin{equation}\begin{split}\label{REG_NEUTRAL_TS}
\rhobar(x,y,z):=
\begin{cases}
& \int_\Omega  j_\eta (x-\tilde{x},y-\tilde{y},z-\tilde{z})\rho(\tilde{x},\tilde{y},\tilde{z})\, d\tilde{x}d\tilde{y}d\tilde{z},\ \text{ for } 
(x,y,z)\in\Omegainthalf,\\
& 0,\quad \text{ else}.
\end{cases}
\end{split}\end{equation}
\noindent $ii)$ We define the {\bf isoneutral density slope vector} $\SL:=\SL(\rhobar):=(\SLx(\rhobar), \SLy(\rhobar))$ {\it with respect to the density} 
$\rhobar$ by 
\begin{equation}\begin{split}\label{SLOPE_DEF_1}
\SL(x,y,z,t):=\SL(\rhobar)(x,y,z,t):=
\begin{cases}
&-\frac{\nabla_h\rhobar(x,y,z,t)}{\partial_z\rhobar(x,y,z,t)}\pi_{\partial\Omega}\Pi_{s_0,\epsilon_0}\big(|\partial_z\rhobar(x,y,z,t)|\big),\\  
&\text{for  }(x,y,z)\in\Omega\setminus N(\partial_z\rhobar),\\
&0\ \text{else, }
\end{cases}               
\end{split}\end{equation}
with $N(\partial_z\rhobar):=\{(x,y,z)\in\Omega: |\partial_z\rhobar(x,y,z)|\leq s_0-\epsilon_0\}$, $s_0>\epsilon_0>0$, 
and $\Pi_{s_0,\epsilon_0}:[0,\infty)\to [0,1]$ a monotonic non-decreasing $C^\infty$-function that 
satisfies 
\begin{equation}\begin{split}\label{SLOPE_DEF_1_TAPERING}
\Pi_{s_0,\epsilon_0}(x)=
\begin{cases}
&1,\quad\text{for}\quad x\geq s_0,\\
&0,\quad\text{for}\quad 0\leq x\leq s_0-\epsilon_0.
\end{cases}
\end{split}\end{equation}
The numbers $s_0, \epsilon_0$ denote the minimal required stratification and width of the transition zone in which we 
reduce the density slope towards zero if the stratification is locally less than $s_0$. \\
\end{definition}

\subsection{Eddy-Induced Diffusion and Advection of Tracers}\label{SUBSECTION_ISONEUTRAL_DIFFUSION}
In this section we introduce the eddy operators $\REDIOP$ and $\GMOP$ that define eddy-induced diffusion and advection.
\begin{definition}[Eddy Operators]\label{REDI_DEF}
Let $\theta\in \Hone\cap\LinftyT, S\in \Hone\cap\LinftyS$ be admissible temperature and salinity fields
with associated isoneutral density slope vector $\SL=\SL (\rhobar)=(\SLx (\rhobar), \SLy (\rhobar))$, defined in (\ref{SLOPE_DEF_1}).\\
$i)$ The {\it isoneutral diffusion operator} $\REDIOP(\rhobar)$ is for $C\in\{\theta, S\}$ defined by
\begin{align}
&\REDIOP(\rhobar) (C):=-\nabla_3\cdot\big(\KISO(\rhobar)\nabla_3 C\big),\label{REDIOP}\\
\text{with }\ &\KISO(\rhobar)
:= \frac{K_I}{1+\SL^2}
\begin{pmatrix}
1+\delta\SLx^2+\SLy^2& (\delta-1)\SLx \SLy&(1-\delta)\SLx\\
(\delta-1)\SLx\SLy& 1+\SLx^2+\delta\SLy^2&(1-\delta)\SLy\\
(1-\delta)\SLx&(1-\delta)\SLy&\delta+\SL^2
\end{pmatrix}, \label{REDI_TENSOR_FULL}
\end{align}
where $\delta:=\frac{K_D}{K_I}<<1$ is the fraction of isoneutral and dianeutral mixing coefficients $K_D, K_I>0$ and 
$\SL^2:=\SL\cdot\SL$ .\\
$ii)$ The operator of {\it eddy-induced advection} $\GMOP$
 is for $C\in\{\theta, S\}$ defined by
\begin{align}
&\GMOP(\rhobar)(C):=\nabla_3\cdot \big(\KGM(\rhobar)\nabla_3 C\big)\label{GM_DEF_1}\\
\text{with }\ 
&\KGM(\rhobar) 
:= \kappa
\begin{pmatrix}
0& 0&-\SLx\\
0& 0&-\SLy\\
\SLx& \SLy&0
\end{pmatrix}\quad \kappa>0.\label{GM_DEF_2}
\end{align}
The {\it eddy-induced transport velocity} is defined as
\begin{equation}\begin{split}\label{GM_BOLUS}
\vbolus:= -\partial_z(\kappa \SL),\text{ and } \wbolus:= \nabla\cdot(\kappa \SL)\ \text{on } \Omegaint.
\end{split}\end{equation}
\end{definition}
\begin{remark}
If the density slope $\SL(x,y,z,t)$ is reduced to zero then the isoneutral diffusion operator $\REDIOP$ locally reduces to an anisotropic Laplacian operator with horizontal diffusivity $K_I>0$, vertical diffusivity $K_D>0$, the eddy advection operator $\GMOP$ vanishes.  
\end{remark}
\begin{remark}[Flux Balance]
The purely isoneutral flux $\mathcal{F}_{iso}(C):=\KISO(\rhobar)\nabla_3 C$ with $K_D=0$,
maintains the balance $\alpha\mathcal{F}_{iso}(C)(\theta)=\beta\mathcal{F}_{iso}(S)$. This balance avoids
spurious buoyancy fluxes. Preserving this balance on the discrete  level is essential for discretizing the isoneutral operator in a physically consistent manner
 and prevents a numerical instability  (cf. \cite{KORN_PARAM}, \cite{GRIFFIES_1}). 
The calculation of the isoneutral flux in Definition \ref{REG_NEUTRAL} preserve this flux balance. 
\end{remark} 
\begin{remark}[EddyAdvection]
The advective part of the eddy parametrization approximates the additional eddy transport of oceanic tracers in terms of the {\it ``eddy-induced transport velocity''} 
$v_3^\circ=(v^\circ,w^\circ)$ such that 
$\overline{v'C'}\sim v_3^\circ C$,
where the overbar denotes the mean and the primed variables the fluctuations. 
The eddy-induced velocity $v_3^\circ$
is determined from a stream function $\Psi=(\Psi^x,\Psi^y,0)$ by
$v_3^\circ=curl\Psi$ and $\Psi:=\vec{z}\times \kappa\SL$.
A direct calculation shows that 
$v^\circ\cdot\nabla_3 C=-\nabla_3\cdot(\KGM\nabla_3 C)$.
\end{remark}
\subsection{Small Density Slope Approximation}\label{Subsect_SMALL_SLOPE}
In simulations of the real ocean, numerical circulation models use for efficiency reasons the ``small slope approximation'' of isoneutral diffusion (see \cite{GENT_MCWILLIAMS})
\begin{equation}\begin{split}\label{REDI_TENSOR_SMALL}
\KISOSMALL
:= K_I
\begin{pmatrix}
1 & 0 &\SLxsmall\\
0 & 1& \SLysmall\\
\SLxsmall& \SLysmall& \delta+|\SLsmall |^2
\end{pmatrix},\quad
\KGMsmall
:= \kappa
\begin{pmatrix}
0& 0&-\SLxsmall\\
0& 0&-\SLysmall\\
\SLxsmall& \SLysmall&0
\end{pmatrix},
\end{split}\end{equation}
where for a chosen parameter $r<<1$ the ``small density slope'' is defined by
\begin{equation}\begin{split}\label{REDI_TENSOR_SMALL_SLOPE}
\SLsmall:=
\begin{cases}
&\SL\quad\text {if  }|\SL|< r,\\
&0\quad\text{else}.
\end{cases}
 \end{split}\end{equation}
This approximation is justified by the observation that in the world ocean 
horizontal density gradients are generally smaller than vertical density gradients, therefore the off-diagonal terms in (\ref{REDI_TENSOR_FULL}) can be neglected (see \cite{GRIFFIES_BOOK}, Chapter. 14).

We define the eddy operators in the small slope approximation accordingly as
\begin{equation}\begin{split}\label{EDDY_OPERATORS_SMALL}
\REDIOPSMALL(\rhobar) (C):=\nabla_3\cdot(\KISOSMALL\nabla_3 C),\quad \GMOPSMALL(\rhobar)(C):=\nabla_3\cdot(\KGMsmall\nabla_3 C).
 \end{split}\end{equation}

\subsection{Properties of the Eddy Parametrization}\label{SUBSECT_PROOF_EDDY_PARAM}

\begin{lemma}\label{LEMMA_DIFFERENCES}
Let $\theta\in L^2_{loc}([0,T],\VT\cap\LinftyT), S\in\L^2_{loc}([0,T],\VS\cap\LinftyS)$ and 
let $\SL$ be the isoneutral density slope, defined in (\ref{SLOPE_DEF_1}).\\
$i)$ The density slope satisfies $\SL\in L^\infty([0,T], H^2(\Omega))$, and
\begin{equation}\begin{split}\label{LEMMA_SLOPE_BOUND}
||\SL||_{H^2(\Omegaint)}\leq ||\rho||^2_\Lsq.
\end{split}\end{equation}
$ii)$  Let $(v,\theta,S)$ be a strong solution of (\ref{STRONG_PE_OCEAN}) (cf. Def. \ref{DEF_STRONGSOLUTION} $i)$. 
Then 
$\partial_t\SL\in L^\infty([0,T], L^2(\Omega))$.\\ 
 $iii)$ Let two pairs of temperature and salinity fields $(\theta_1,S_1)$, $(\theta_2,S_2)\in L^2_{loc}([0,T],\VT\cap\LinftyT) \times 
L^2_{loc}([0,T],\VS\cap\LinftyS)$ be given. 
For two density slopes $\SL(\rhobar_1), \SL(\rhobar_2)$ it holds
\begin{align}
&||\SL(\rhobar_1)(t)-\SL(\rhobar_2)(t)||_{L^\infty(\Omega)}
\leq 
c\tau(\frac{c}{s_0}+\frac{C}{s_0^2})|I_\rho|||\rho_1(t)-\rho_2(t)||_{\Lsq},
\end{align}
\end{lemma}
\begin{proof}
$Ad\, i)$ We decompose the domain 
 $\Omegaint= I_1\cup I_2\cup I_3$,
 where $I_1:=\{(x,y,z)\in\Omegaint:\, |\partial_z\rhobar(x,y,z)|\geq s_0 \}$, 
 $I_2:=\{(x,y,z)\in\Omegaint:\, s_0-\epsilon\leq |\partial_z\rhobar(x,y,z)|\leq s_0 \}$,
$I_3:=\{(x,y,z)\in\Omegaint:\, 0\leq |\partial_z\rhobar(x,y,z)|< s_0-\epsilon \}$. These sets are measurable, since $\partial_z\rhobar$ is continuous. Using this decomposition we find with $l\in \mathbb{N}^3$,
\begin{equation}\begin{split}\label{COROLLARY_SLOPES_1_PF}
||\SL||_{H^2(\Omegaint)}^2
&=\sum_{|l|\leq 2}
\int_{I_1} |D^l  \big(\frac{\nabla_h\rhobar }{\partial_z\rhobar} \big)|^2\, dxdydz
+\int_{I_2} |D^l  \big(\frac{\nabla_h\rhobar}{\partial_z\rhobar}\Pi_{s_0,\epsilon_0}(|\partial_z\rhobar|)\big)|^2\, dxdydz,
\end{split}\end{equation}
where the integral over $I_3$ vanishes due to the tapering by $\Pi_{s_0,\epsilon_0}$. 
For the first integral 
we obtain with the Leibniz rule, the commutation of convolution and differentiation for $k\in\mathbb{N}^3$
\begin{equation}\begin{split}\label{COROLLARY_SLOPES_1_PF_1}
&\int_{I_1} |D^l\SL|^2\, dxdydz
=\sum_{|k|\leq l}\int_{I_1} |D^k (\nabla_h\rhobar ) D^{l-k}( (\partial_z\rhobar)^{-1})|^2\, dxdydz\\
&\leq c\sum_{|k|\leq l}\int_{I_1} |(\rho\ast D^{k+1}j_\eta)( \rho\ast D^{l-k+1}j_\eta)|^2\, dxdydz\leq c||\rho||_\Lsq^4.
\end{split}\end{equation}
An analogous estimate follows for the second integral in (\ref{COROLLARY_SLOPES_1_PF}), with an upper bound $c>0$ that depends in addition on $\Pi_{s_0,\epsilon_0}$.\\
$Ad\, ii)$ For the $L^2$-norm of the time derivative follows with $l\in \mathbb{N}^3$, 
and by using the same decomposition of $\Omegaint= I_1\cup I_2\cup I_3$ as in part $i)$
\begin{equation}\begin{split}\label{COROLLARY_SLOPES_1_PF_2}
||\partial_t\SL||_{L^2(\Omegaint)}^2
&=
\int_{I_1} |\partial_t  \big(\frac{\nabla_h\rhobar }{\partial_z\rhobar} \big)|^2\, dxdydz
+\int_{I_2} |\partial_t  \big(\frac{\nabla_h\rhobar}{\partial_z\rhobar}\Pi_{s_0,\epsilon_0}(|\partial_z\rhobar|)\big)|^2\, dxdydz.
\end{split}\end{equation}
For the first integral on the right-hand side we obtain with 
\begin{equation}\begin{split}\label{COROLLARY_SLOPES_1_PF_3}
&
\int_{I_1} | \partial_t  \big(\frac{\nabla_h\rhobar }{\partial_z\rhobar} \big)|^2\, dxdydz
=\int_{I_1} |\frac{\partial_z\rhobar \partial_t  \nabla_h\rhobar - \nabla_h\rhobar\partial_t\partial_z\rhobar }
{|\partial_z\rhobar|^2}
 |^2\, dxdydz
 \leq c\int_{I_1}|g|^2\, dxdydz
\end{split}\end{equation}
with 
$g:=\partial_z\rhobar \partial_t  \nabla_h\rhobar - \nabla_h\rhobar\partial_t\partial_z\rhobar$.
For the function $g$ in (\ref{COROLLARY_SLOPES_1_PF_3}) we obtain 
\begin{equation}\begin{split}\label{COROLLARY_SLOPES_1_PF_3a}
&|g|=|\partial_{x_j}\rhobar \partial_t  \partial_{x_i}\rhobar|
=
|\big( \rho\ast J_1^{j}\big)\big(  \partial_t\rho\ast J_1^{x_i}\big)|
\leq
c 
|\partial_t\rho\ast J_1^{x_i}|
=|(aF_\theta+bF_S)\ast J_1^{x_i} |,
\end{split}\end{equation}
with $J_1^{x_i}:=\partial_{x_i}j_\eta\in C^\infty(\Omega)$, ($i=1,2,3$) 
and where we have used 
the evolution equation (\ref{density_eq}) for density. The thermal expansion function ``a'' and the saline contraction function ``b'' are defined in (\ref{density_eq}).
For $a\,F_\theta$ follows with the boundedness of the thermal expansion function 
\begin{equation}\begin{split}\label{COROLLARY_SLOPES_1_PF_4a}
& \int_{I_1}|(aF_\theta)\ast J_1^{x_i} |^2\, dxdydz
= \int_{I_1} |\big(a\, div_3(v \theta+(\KISO(\rhobar)+\KGM(\rhobar))\nabla_3\theta)\big)\ast J_1^{x_i} |^2\, dxdydz\\
&\leq ||a||_{L^\infty(\Omega)}^2\int_{I_1} | \big(v \theta+(\KISO(\rhobar)+\KGM(\rhobar))\nabla_3\theta\big)\ast 
\nabla_3 J_1^{x_i} |^2\, dxdydz\leq M_\theta,
\end{split}\end{equation}
here $M_\theta(t)$ denotes a continuous function that dominates the integral in (\ref{COROLLARY_SLOPES_1_PF_4a}),
because the integrand 
is a convolution of an integrable function with a regularizing kernel and therefore smooth and bounded. 
Analogously follows for $bF_S$ in (\ref{COROLLARY_SLOPES_1_PF_3a}) 
that 
$ \int_{I_1} |bF_S\ast J_1^{n-l} |^2\, dxdydz\leq M_S$.
The function $M:=M_\theta+M_S$ bounds the first term on the right-hand side of (\ref{COROLLARY_SLOPES_1_PF_2}). The bound for the second term follows analogously.\\ 
$Ad\, iii)$ 
From 
the continuity of $\Pi_{s_0,\epsilon_0}$,
and the regularity properties of $\rhobar$ follows
\begin{equation*}\begin{split}
&|\SL_2-\SL_1|
\leq
(\frac{c}{s_0}+\frac{C}{s_0^2})||\nabla_h\rhobar_1||_{\Htwoint}||\partial_z(\rhobar_1-\rhobar_2)||_{\Htwoint}\\
&+\frac{C}{s_0^2}
||\partial_z\rhobar_1||_{\Htwoint}||\nabla(\rhobar_1-\rhobar_2)||_{\Htwoint}
\leq
c\tau(\frac{c}{s_0}+\frac{C}{s_0^2})|I_\rho|||\rho_1(t)-\rho_2(t)||_{\Lsq}.
\end{split}\end{equation*}
\end{proof}
\begin{remark}\label{remark_three_derivatives}
Lemma \ref{LEMMA_DIFFERENCES} is used in Theorem \ref{THEOREM_EXISTENCE_WEAKSTRONG_SOLUTIONS} on the existence of weak solutions,
where we use it to prove the convergence of the 
Galerkin approximation of the eddy operators. We need this result also for the elliptic regularity of $\REDIOP$ in Lemma \ref{LEMMA_ELLIPTIC_REG_REDI}.
In our main result, Theorem \ref{THEOREM_EXISTENCE_STRONG_SOLUTIONS}, we apply the lemma to establish the $H^1$-estimate for 
temperature and salinity, to prove the estimate on the time derivative of these tracers for the Aubin-Lions lemma and in the uniqueness part of the proof. 
\end{remark}
\begin{lemma}[Ellipticity of $\KISO$]\label{LEMMA_PROP_REDIOPERATOR}
Let a isoneutral density slope $\SL_0:=\SL(\rhobar_0)$, defined in (\ref{SLOPE_DEF_1}), be given. 
The isoneutral operator $\REDIOP$, defined in (\ref{REDIOP}), satisfies for $\theta\in \Hone\cap\LinftyT, S\in\Hone\cap\LinftyS$ the following inequalities
\begin{align}
&\mu||\nabla_3\theta||_{\Lsq}^2+k_\theta\int_{\Gamma_u}|\theta(x,y,0)|^2\, dxdy\nonumber\\
&\leq\int_\Omega \REDIOP(\rhobar_0)(\theta)\theta\, dxdydz
 \leq M||\nabla_3\theta||_{\Lsq}^2+k_\theta\int_{\Gamma_u}|\theta(x,y,0)|^2\, dxdy\nonumber\\
\text{and }\  &
\mu||\nabla_3 S||_{\Lsq}^2\leq\int_\Omega \REDIOP(\rhobar_0)(S)S\, dxdydz \leq M||\nabla_3 S||_{\Lsq}^2,\nonumber\\
\text{with }\  &\mu:=min\{K_I,K_D\}, M:=max\{K_I,K_D\}.
 \label{LEMMA_PROP_REDIOPERATOR_3}
\end{align}
\end{lemma}
\begin{proof}
The isoneutral mixing operator $\KISO$ can be written as
\begin{equation}\begin{split}\label{PROP_MIXING_TENSOR_2}
\KISO(\rhobar)=\mathcal{R}(\SL) \mathbb{D}\mathcal{R}^T(\SL),
\end{split}\end{equation}
where the orthogonal matrix $\mathcal{R}(\SL)$ and the diagonal matrix $\mathbb{D}$ are given by
\begin{equation}\begin{split}\label{PROP_MIXING_TENSOR_3}
\mathcal{R}(\SL):=
\begin{pmatrix}
\frac{\SLy}{|\SL|}&\frac{\SLx}{|\SL|\sqrt{1+\SL^2}}&-\frac{\SLx}{\sqrt{1+\SL^2}}\\
-\frac{\SLx}{|\SL|}&\frac{\SLy}{|\SL|\sqrt{1+\SL^2}}&-\frac{\SLy}{\sqrt{1+\SL^2}}\\
0&\frac{\SL}{\sqrt{1+\SL^2}}&\frac{1}{\sqrt{1+\SL^2}}
\end{pmatrix},\qquad
\mathbb{D}=
\begin{pmatrix}
K_I&0&0\\
0&K_I&0\\
0&0&K_D
\end{pmatrix}.
\end{split}\end{equation}
For (\ref{LEMMA_PROP_REDIOPERATOR_3}) follows after integration by parts and with (\ref{PROP_MIXING_TENSOR_2})
\begin{equation*}\begin{split}\label{PROOF_POS_DEF_FULL}
&\int_\Omega \REDIOP(\rhobar_0)(C)C\, dxdydz
=\int_\Omega |\mathbb{D}^{1/2}\mathcal{R}^T(\SL_0)\nabla_3C|^2\, dxdydz\\
&+k_C\int_{\Gamma_u}|C(x,y,0)|^2\, dxdy\\
&=\int_\Omega K_I|(\mathcal{R}^T(\SL_0)\nabla_3C)\cdot e_1|^2+K_I|(\mathcal{R}^T(\SL_0)\nabla_3C)\cdot e_2|^2\\
&+K_D|(\mathcal{R}^T(\SL_0)\nabla_3C)\cdot e_3|^2
\, dxdydz
+k_C\int_{\Gamma_u}|C(x,y,0)|^2\, dxdy, 
\end{split}\end{equation*}
where $e_i$, $i=1,2,3,$ denotes the canonical basis in $\mathbb{R}^3$ and $\mathbb{D}^{1/2}:=diag(\sqrt{K_I},\sqrt{K_I},$ $\sqrt{K_D})$ the diagonal matrix. Since $\mathcal{R}^T(\SL_0)$ is orthogonal, the assertion follows.
\end{proof}


\begin{lemma}[Elliptic Estimate for Isoneutral Operator]\label{LEMMA_ELLIPTIC_REG_REDI}
Let $F, G\in \Lsq$ be given. 
Then there exists a solution $(\theta, S)$ of the equation system for the isoneutral diffusion operator $\REDIOP$ (cf. (\ref{REDIOP}))
\begin{equation}\begin{split}\label{REDI_OP_SOLVE}
&\REDIOP(\rhobar)(\theta)=F,\quad \text{ and }\quad
\REDIOP(\rhobar)(S)=G,
\end{split}\end{equation}
with boundary conditions (\ref{PE_AO_BOUNDARY_CONDITIONS_TRAC2}).
Equation (\ref{REDI_OP_SOLVE}) is satisfied 
in the following sense for all $\phi^\theta\in\VT$ and $\phi^S\in\VS$
\begin{equation}\begin{split}\label{REDI_OP_SOLVE2}
&\int_\Omega(\KISO(\rhobar)\nabla_3\theta)\cdot\nabla_3\phi^\theta\,dxdydz+k_\theta\int_{\Gamma_u}\theta\phi^\theta\,dxdy
=\int_\Omega F\phi^\theta\,dxdydz,\\
&\int_\Omega(\KISO(\rhobar)\nabla_3 S)\cdot\nabla_3\phi^S\,dxdydz
=\int_\Omega G\phi^S\,dxdydz,
\end{split}\end{equation}
where $\KISO$ is defined in (\ref{REDI_TENSOR_FULL}).
The solution $(\theta,S)$ satisfies
\begin{equation}\begin{split}\label{REDI_OP_SOLVE3}
||\theta||_{{\Htwo}}\leq c||\REDIOP(\rhobar)(\theta)||_{\Lsq},\quad\text{and}\quad ||S||_{{\Htwo}}\leq c||\REDIOP(\rhobar)(S)||_{\Lsq},
\end{split}\end{equation}
where $c=c(K_I, K_D, |\Omega|, |I_\rho|)$, and $|I_\rho|$  is the interval length of admissible density values.
\end{lemma}
\begin{proof}
{\bf Step 1: {\it Solvability of Linearized Equation.}}\\
We linearize (\ref{REDI_OP_SOLVE2}) by fixing an arbitrary pair $(\theta^{'},S^{'})\in\VT,\times\VS$. With the associated density 
$\rhobar^{'}:=\rhobar^{'}({\theta}^{'},{S}^{'})$ 
we define $D_{\KISO}(\rhobar^{'})(C):=\nabla_3\cdot\big(\KISO({\rhobar}^{'})\nabla_3 C\big)$.
This operator is linear, continuous and uniformly elliptic on ${\VT}\times{\VS}$ (cf. Prop. \ref{LEMMA_PROP_REDIOPERATOR}). 
The linearized equation on $\VT\times\VS$ reads
\begin{equation}\begin{split}\label{REDI_OP_SOLVE_GALERKIN_12}
&\int_\Omega(\KISO({\rhobar}^{'})\nabla_3\theta)\cdot\nabla_3\phi^\theta\,dxdydz
+k_\theta\int_{\Gamma_u}\theta\phi^\theta\,dxdy
=\int_\Omega F\phi^\theta\,dxdydz,\\
&\int_\Omega(\KISO({\rhobar}^{'})\nabla_3 S)\cdot\nabla_3\phi^S\,dxdydz
=\int_\Omega G\phi^S\,dxdydz.
\end{split}\end{equation}
The Lax-Milgram Lemma shows that for all fixed $\theta^{'}\in\VT,{S}^{'}\in \VS$ equation (\ref{REDI_OP_SOLVE_GALERKIN_12}) has a unique weak solution $\theta=\theta({\theta}^{'})\in\VT,\, S=S({S}^{'})\in\VS$, which according to elliptic regularity theory (e.g. Sect. 9.6. in \cite{BREZIS})
satisfies
\begin{equation}\begin{split}\label{REDI_OP_SOLVE_GALERKIN_13}
||\theta||_{{H^2(\Omega)}}\leq  c||F||_{\Lsq}\quad \text{and}\quad ||S||_{{H^2(\Omega)}}\leq  c||G||_{\Lsq},
\end{split}\end{equation}
here $c$ depends on the maximum norm $||\nabla_3k_{i,j}(\rhobar^{'})||_{L^\infty}$ of the gradient of the matrix elements 
$k_{i,j}(\rhobar^{'})$ of $\KISO$. 
Since $||\nabla_3k_{i,j}(\rhobar^{'})||_{L^\infty}\leq c_0||\rho'||_\Lsq\leq c_0|\Omega|\, |I_\rho|$, with $c_0$ depending on constants $K_I, K_D$, we conclude that $c=c(K_I, K_D, |\Omega|, |I_{\rho}|)$. In particular is $c$ independent of $(\theta^{'},S^{'})\in\VT,\times\VS$.\\
{\bf Step 2: {\it Solvability of System (\ref{REDI_OP_SOLVE2}).}}\\
Consider the mapping $H$ that assigns 
to each fixed pair $({\theta}^{'},{S}^{'})\in \VT\times \VS$
a solution $(\theta,S)\in {\Htwo}\times {\Htwo}$ of (\ref{REDI_OP_SOLVE_GALERKIN_12})
with coefficient matrix $\KISO(\rhobar^{'})$
\begin{equation}\begin{split}\label{REDI_OP_SOLVE_GALERKIN_14}
H: \VT\times \VS &\to \VT \times \VS\\
({\theta}^{'},{S}^{'})&\mapsto H({\theta}^{'},{S}^{'}):=(\theta,S).
\end{split}\end{equation}
We assert the following three properties of $H$
\begin{equation}\begin{split}\label{REDI_OP_SOLVE_GALERKIN_14b}
i)\ &H\ \text{maps a closed ball}\ B_R\subset\VT\times \VS \text{ into itself}.\\
ii)\ &H\ \text{is continuous,}  \\
iii)\ &H\ \text{is compact.} 
\end{split}\end{equation}
$Ad\, i)$ Consider the ball $B_{R_1}^\theta\subset \VT$ and $B_{R_2}^S\subset \VS$ with radius $R_1:=c||F||_{\Lsq}$ and 
$R_2:=c||G||_{\Lsq}$, with $F,G$ and  $c=c(K_I, K_D, |\Omega|, |I_{\rho}|)$ from (\ref{REDI_OP_SOLVE_GALERKIN_13}).  
Choose a non-zero $(\theta_0,S_0)\in\VT\times \VS$, define 
${\theta}^{'}:=R_1\theta_0/||\theta_0||_{\Hone}$ and ${S}^{'}:=R_2S_0/||S_0||_{\Hone}$. Then 
$({\theta}^{'},{S}^{'})\in B_R:=B_{R_1}^\theta\times B_{R_2}^S$.
From (\ref{REDI_OP_SOLVE_GALERKIN_13}) follows that 
$H({\theta}^{'},{S}^{'})=(\theta,S)$ 
satisfies $||\theta||_{\Htwo}\leq R_1, ||S||_{{\Htwo}}\leq R_2$. 
Thus $||\theta||_{\Hone}\leq R_1$ and $||S||_{\Hone}\leq R_2$. \\ 
$Ad\, ii)$ 
Let the sequence
$({\theta}_k^{'},{S}_k^{'})_k\subseteq \VT\times\VS$ converge with respect to the $H^1$-topology 
of $\VT\times\VS$ to a limit $({\theta}^{'},{S}^{'})\in \VT\times\VS$ and denote the associated densities
by $(\rhobar_k^{'})_k$ and $\rhobar{'}$. 
For each element of $({\theta}_k^{'},{S}_k^{'})_k$  exists a solution 
$(\theta_k, S_k)_k:=(H({\theta}_k^{'},{S}_k^{'} ))_k$ $\subseteq \VT\times\VS$ of  (\ref{REDI_OP_SOLVE_GALERKIN_12}), 
 i.e. for  $\phi^\theta\in\VT^m, \phi^S\in\VS^m$,  $k\in\mathbb{N}$ we have
\begin{equation}\begin{split}\label{REDI_OP_SOLVE_GALERKIN_14c}
&\int_\Omega(\KISO({\rhobar}_k^{'})\nabla_3\theta_k)\cdot\nabla_3\phi^\theta\,dxdydz+k_\theta\int_{\Gamma_u}\theta_k\phi^\theta\,dxdy
=\int_\Omega F\phi^\theta\,dxdydz,\\
&\int_\Omega(\KISO({\rhobar}_k^{'})\nabla_3 S_k)\cdot\nabla_3\phi^S\,dxdydz=\int_\Omega G\phi^S\,dxdydz.
\end{split}\end{equation}
For equation (\ref{REDI_OP_SOLVE_GALERKIN_12}) with a coefficient matrix $\KISO(\rhobar^{'})$ given by the limit 
$\rhobar^{'}$ of $(\rhobar_k)_k$ there exists a 
solution $(\theta^{\sharp}, S^{\sharp})\in \VT\times\VS$
such that for $\phi^\theta\in\VT, \phi^S\in\VS$
\begin{equation}\begin{split}\label{REDI_OP_SOLVE_GALERKIN_14d}
&\int_\Omega(\KISO({\rhobar}^{'})\nabla_3\theta^{\sharp})\cdot\nabla_3\phi^\theta\,dxdydz
+k_\theta\int_{\Gamma_u}\theta^{\sharp}\phi^\theta\,dxdy
=\int_\Omega F\phi^\theta\,dxdydz,\\
&\int_\Omega(\KISO({\rhobar}^{'})\nabla_3 S^{\sharp})\cdot\nabla_3\phi^S\,dxdydz
=\int_\Omega G\phi^S\,dxdydz.
\end{split}\end{equation}
We show now that the sequence of solutions $(\theta_k,S_k)$ of (\ref{REDI_OP_SOLVE_GALERKIN_14c}) converges to the solution
$(\theta^{\sharp}, S^{\sharp})$ of (\ref{REDI_OP_SOLVE_GALERKIN_14d}). 
After subtracting (\ref{REDI_OP_SOLVE_GALERKIN_14d}) from (\ref{REDI_OP_SOLVE_GALERKIN_14c}) we obtain for the temperature difference with $\widehat{\mathbb{K}}_{iso}:= \KISO(\rhobar_k^{'})- \KISO(\rhobar^{'})$ and for $\phi^\theta\in\VT$ that
\begin{align}
0&=\int_\Omega\big(\KISO({\rhobar}_k^{'})\nabla_3\theta_k-\KISO({\rhobar}^{'})\nabla_3\theta^{\sharp}\big)\cdot\nabla_3\phi^\theta\,dxdydz
+k_\theta\int_{\Gamma_u}(\theta_k-\theta^{\sharp})\phi^\theta\,dxdy\nonumber\\
&=\int_\Omega\big( \widehat{\mathbb{K}}_{iso} \nabla_3\theta_k
+ \KISO({\rhobar}^{'})\nabla_3(\theta_k-\theta^{\sharp})\big)\cdot\nabla\phi^\theta\, dxdydz
+k_\theta\int_{\Gamma_u}(\theta_k-\theta^{\sharp})\phi^\theta\,dxdy.\label{REDI_OP_SOLVE_GALERKIN_14e}
\end{align}
In the first term on the right-hand of (\ref{REDI_OP_SOLVE_GALERKIN_14e}), the differences $\widehat{\mathbb{K}}_{iso}=(\hat{k}_{i,j})_{i,j}$ can be bounded by density slope differences, i.e $|\hat{k}_{i,j}|\leq c|\SL(\rhobar_k^{'})-\SL(\rhobar^{'})|$. Lemma \ref{LEMMA_DIFFERENCES} bounds slope differences by density differences, $|\SL(\rhobar_k^{'})-\SL(\rhobar^{'})|\leq c|\rhobar_k^{'}-\rhobar^{'}|$, and Lemma \ref{LEMMA_DENSITY_DIFF} bounds density differences by temperature, salinity differences and with the $H^1$-convergence of $(\theta_k^{'}, S_k^{'})$ to $(\theta^{'},S^{'})$
follows
\begin{equation}\begin{split}\label{REDI_OP_SOLVE_GALERKIN_14g}
&\lim_{k\to\infty}\int_\Omega\big( \widehat{\mathbb{K}}_{iso} \nabla_3\theta_k\big)\cdot\nabla_3\phi^\theta\, dxdydz\\
&\leq c\lim_{k\to\infty}||\theta_k||_{{\Hone}}||\phi^\theta||_{\Hone} \big(||\theta{'}-\theta_k^{'}||_{\Lsq)} 
+||S^{'}-S_k^{'}||_{\Lsq)}\big)=0.
\end{split}\end{equation}
For the second term on the right-hand side of (\ref{REDI_OP_SOLVE_GALERKIN_14e}) we derive with the ellipticity of $\KISO$ (cf. Lemma \ref{LEMMA_ELLIPTIC_REG_REDI}) 
and with the particular choice $\phi^\theta:=\theta_k-\theta^{\sharp}$
\begin{equation}\begin{split}\label{REDI_OP_SOLVE_GALERKIN_14h}
0\geq& \lim_{k\to\infty}\big|\int_\Omega\big(  \KISO({\rhobar}^{'})\nabla_3(\theta_k-\theta^{\sharp})\big)\cdot\nabla_3(\theta_k-\theta^{\sharp})\, dxdydz
+k_\theta\int_{\Gamma_u}|\theta_k-\theta^{\sharp}|^2\,dxdy\big|\\
\geq &\mu\lim_{k\to\infty}\big\{||\nabla_3(\theta_k-\theta^{\sharp})||_{\Lsq}^2+k_\theta||\theta_k(z=0)-\theta^{\sharp}(z=0)||^2_{\Lsq}\big\}.
\end{split}\end{equation}
From the convergence of $(\nabla_3\theta_k)_k$ to $\nabla\theta^{\sharp}$ in $\Lsq$ and of $(\theta_k(z=0))_k$ in $L^2(\Gamma_u)$ 
to $\theta^{\sharp}(z=0)$ follows with Poincar\'e's inequality (cf. Prop. III.2.38 in \cite{Boyer_NSE_book}) that $(\theta_k)_k $ converges to 
$\theta^{\sharp}$ in $\Hone$. 
Thus $(\theta_k)_k$ converges to $\theta^{\sharp}$ in $\Hone$. 
For salinity we obtain an analogous statement to (\ref{REDI_OP_SOLVE_GALERKIN_14h}), 
here the boundary term is absent due to the boundary condition for salinity. 
From 
Poincar\'e's  inequality  (cf. Prop. III.2.39 in \cite{Boyer_NSE_book}) follows that $(S_k)_k$ converges to $S^{\sharp}$ in $\Hone$.
Since $(\theta_k, S_k)=H({\theta}_k^{'},{S}_k^{'})$ and $(\theta^\sharp, S^\sharp)=H({\theta}^{'},{S}^{'})$ this proves the continuity of $H$.\\ 
\noindent $Ad\, iii)$
Consider a sequence
$({\theta}_k^{'},{S}_k^{'})_k\subseteq B_{R_1}^\theta\times B_{R_2}^S$. 
The elements of the image sequence under $H$, 
$({\theta}_k,{S}_k):=H({\theta}_k^{'},{S}_k^{'})$, solve  (\ref{REDI_OP_SOLVE_GALERKIN_12})
with coefficients determined by $\rhobar_k^{'}:=\rhobar({\theta}_k^{'},{S}_k^{'})$. This implies
that for all $k\in\mathbb{N}$ the solution $\theta_k,{S}_k$ satisfies  
\begin{equation}\begin{split}\label{REDI_OP_SOLVE_GALERKIN_20}
||\theta_k||_{{\Htwo}}\leq  c||F||_{\Lsq}\quad \text{and}\quad ||S_k||_{{\Htwo}}\leq  c||G||_{\Lsq},
\end{split}\end{equation}
since $c=c(K_I, K_D, |\Omega|, |I_{\rho}|)$ the sequence
 is bounded uniformly in $\Htwo$. This implies that $(\theta_k,S_k)_k$ converges weakly in $\Htwo$.
Due to the compact embedding of $\Htwo$ in $\Hone$ there exists a subsequence $(\theta_n,S_n)_n$ that converges
in $\Hone$. 

From properties $i)-iii)$ in (\ref{REDI_OP_SOLVE_GALERKIN_14b}) follows with the Tikhonnov-Schauder Fixed-Point Theorem that the mapping $H$ has a fix point, i.e. $H(\theta,S)=(\theta,S)$. 
\end{proof}

\section{Regularity Results of Ocean Primitive Equations with Mesoscale Eddy Parametrization}\label{SECT_STRONG_SOLUTION}
We specify now boundary conditions and introduce basic function spaces. In Section \ref{SECT_MAIN_RESULT} we state our main results.

\paragraph*{Boundary Conditions.} 
The domain is $\Omega:=M\times (-h,0)\subseteq \mathbb{R}^3$, with depth $h>0$, the bottom $M\subseteq\mathbb{R}^2$ is a bounded domain in $\mathbb{R}^2$ with a boundary $\partial M$ that is a $C^2$-curve. 
The boundary $\partial \Omega=\Gamma_s\cup\Gamma_b\cup\Gamma_u$ consists of a lateral boundary $\Gamma_s:=\{(x,y,z)\in\bar{\Omega}:(x,y)\in\partial M \}$, a bottom boundary $\Gamma_b:=\{(x,y,z)\in\bar{\Omega}:z=-h \}$ and a surface boundary 
$\Gamma_u:=\{(x,y,z)\in\bar{\Omega}:z=0 \}$. 
The boundary conditions follow Cao-Titi \cite{CaoTiti}. 
Specifically, we impose at the top a wind driven boundary condition for velocity and a rigid-lid for vertical velocity. At the side a no penetration and a stress-free boundary condition, 
at the bottom we employ for the horizontal velocity the stress-free boundary condition for the horizontal and a no penetration boundary condition for the vertical velocity: 
\begin{subequations}\label{PE_AO_BOUNDARY_CONDITIONS_VELOC_1}
\begin{align*}
&on\ \Gamma_u:
\partial_zv=h\tau, w=0,\\
&on\ \Gamma_s:
 \ v\cdot\vec{n}=0, \partial_{\vec{n}} v\times \vec{n}=0, w=0,\ \\
&on\ \Gamma_b:
\partial_z v=0, w=0, 
\end{align*}
\end{subequations}
where $\tau$ is a given 2D wind stress field. 
Since the isoneutral density slope vector $\SL$ vanishes by definition in the vicinity of the boundary $\partial\Omega$
only constant diagonal matrix terms of $\KISO$ remain. 
This allows to express the tracer boundary conditions 
without the mixing tensor $\KISO$  in the following way 
\begin{subequations}\label{PE_AO_BOUNDARY_CONDITIONS_TRAC2}
\begin{align}
on\ \Gamma_u:&\quad \nabla_3\theta\cdot n_3=- k_\theta(\theta-\theta^*),\ \text{and } \nabla_3 S\cdot n_3=0, \label{BC_tracer_top}\\
on\ \Gamma_s: &\quad \nabla_3\theta\cdot \vec{n}=\nabla_3 S\cdot \vec{n}=0, \label{BC_tracer_lateral}\\
on\ \Gamma_b:&\quad \nabla_3\theta\cdot n_3=\nabla_3 S\cdot n_3=0 .\label{BC_tracer_bottom}
\end{align}
\end{subequations}
The boundary conditions for velocity and potential temperature can be homogenized (see \cite{LIONS_TEMAM_WANG_2}, Sect. 2.4., \cite{CaoTiti} p. 248) by adding a $\tau$- and $\theta^*$-dependent term such that
we can assume without loss of generality that $\tau=0$ and $\theta^*=0$.
\paragraph*{Function Spaces.}
We define the spaces for velocity, temperature and salinity. 
\begin{equation*}\begin{split}
\tilde{\HV}:=&\bigg\{v\in ({\bf L}^2(\Omega))^2 :\, \nabla_h\cdot\int_{-h}^0 v(z)\, dz=0, v\text{ satisfies boundary condition } (\ref{PE_AO_BOUNDARY_CONDITIONS_VELOC_1})  \bigg\},\\
\tilde{\HT}:=&\bigg\{\theta\in \Lsq:\theta\text{ satisfies boundary condition }
(\ref{PE_AO_BOUNDARY_CONDITIONS_TRAC2})\text{ for temperature}\bigg\},\\
\tilde{\HS}:=&\bigg\{S\in\Lsq : S\text{ satisfies boundary condition }
(\ref{PE_AO_BOUNDARY_CONDITIONS_TRAC2})\text{ for salinity,} \int_\Omega S dxdydz=0
\bigg\}.\\
\tilde{\VV}:=&\bigg\{v\in ({\bf C}^\infty(\bar{\Omega}))^2 : v\text{ satisfies boundary condition } (\ref{PE_AO_BOUNDARY_CONDITIONS_VELOC_1})  \bigg\},\\
\tilde{\VT}:=&\bigg\{\theta\in C^\infty(\bar{\Omega}) :\theta\text{ satisfies boundary condition }
(\ref{PE_AO_BOUNDARY_CONDITIONS_TRAC2})\text{ for temperature }
\bigg\},\\
\tilde{\VS}:=&\bigg\{S\in C^\infty(\bar{\Omega}) : S\text{ satisfies boundary condition }
(\ref{PE_AO_BOUNDARY_CONDITIONS_TRAC2})\text{ for salinity,}\int_\Omega S\, dxdydz=0\bigg\}.
\end{split}\end{equation*}
Denote by $\VV$, $\VT, \VS$ the closure of $\tilde{\VV}$ in the Sobolev space $(\Hone)^2$ and of $\tilde{\VT}, \tilde{\VS}$ in 
$\Hone$ and by  $\HtwoV, \HtwoT, \HtwoS$ the closure of $\tilde{\VV}$ in $(\Htwo)^2$ and of $\tilde{\VT}, \tilde{\VS}$ in 
$\Htwo$.

The vertical velocity $w$, that is determined by the constraint (\ref{PE_OCEAN_3}), can by virtue of the boundary conditions be expressed as
\begin{equation}\label{VERT_VELOC_REFORM}
w(x,y,z,t)=-\int_{-h}^z\nabla_h\cdot v(x,y,\xi,t)\, d\xi=-\nabla\cdot\int_{-h}^z v(x,y,\xi,t)\, d\xi.
\end{equation}
The pressure term in the velocity equation (\ref{PE_OCEAN_1}) can with the hydrostatic approximation (\ref{PE_OCEAN_2})  be formulated as 
\begin{equation}\label{PRESSURE_REFORM}
p(x,y,z,t)=\int_{-h}^zg\rho(\theta,S, p_{st})(x,y,\xi,t)\, d\xi+p_s(x,y,t),
\end{equation}
where $p_s$ is the surface pressure, $p_{st}$ is the static pressure given by (\ref{GIBBS_PRESSURE_BOUSSINESQ}). 
\begin{definition}[Weak and Strong Solutions]\label{DEF_STRONGSOLUTION}
Let initial conditions $v_0\in \HV$ and 
$\theta_0\in \HT\cap \LinftyT, S_0\in \HS\cap\LinftyS$ be given and $[0,T], {T}>0$ be a time interval. Denote by $\KISO, \KGM$ the isoneutral diffusion 
 and the advection tensor, defined in (\ref{REDI_TENSOR_FULL}) and (\ref{GM_DEF_2}), respectively.\\
$i)$ The triple $(v,\theta,S)$ is called a {\it weak solution}
of (\ref{PE_OCEAN}) on $[0,T]$ if
it satisfies for all testfunctions $\Phi\in \HtwoV$ and $\phi^\theta\in \HtwoT, \phi^S\in \HtwoS$ the equations
\begin{subequations}\label{STRONG_PE_OCEAN}
\begin{align}
&\int_\Omega \partial_t v\cdot\Phi\, dx
+\int_\Omega (v\cdot\nabla) v\cdot\Phi\, dx 
-\int_\Omega \big(\nabla\cdot\int_{-h}^z v(x,y,\xi,t)\, d\xi \big )(\partial_z v)\Phi\, dxdydz \nonumber\\
&+\int_\Omega ( \int_{-h}^zg\rho(x,y,\xi,t)\,d\xi)\nabla_h\cdot \Phi\, dxdydz 
+\int_\Omega f\vec{k}\times v\Phi\, dxdydz\nonumber\\
&+\int_\Omega \frac{1}{Re_1}\nabla v\cdot\nabla\Phi 
+ \frac{1}{Re_2}\partial_z v\partial_z \Phi\, dxdydz
=0,\label{STRONG_PE_OCEAN_1}\\
&\int_\Omega \partial_t \theta\phi^\theta\,dxdydz
+ \int_\Omega (v\cdot\nabla) \theta \phi^\theta\,dxdydz\nonumber\\
&-\int_\Omega \big(\nabla\cdot\int_{-h}^0 v(x,y,\xi,t)\, d\xi \big )(\partial_z \theta)\phi\,dxdydz+k_\theta\int_{\Gamma_u}\theta\phi^\theta\, dxdy\nonumber\\
&+\int_\Omega \big(\KISO(\rhobar)\nabla_3 \theta\big)\cdot\nabla_3\phi^\theta\,dxdydz
+\int_\Omega \big(\KGM(\rhobar)\nabla_3 \theta\big)\cdot\nabla_3\phi^\theta\,dxdydz=0,\label{STRONG_PE_OCEAN_4}\\
&\int_\Omega \partial_t S\psi^S\,dxdydz
+ \int_\Omega (v\cdot\nabla) S\psi^S\,dxdydz\nonumber\\
&-\int_\Omega \big(\nabla\cdot\int_{-h}^0 v(x,y,\xi,t)\, d\xi \big )(\partial_z S)\psi^S\,dxdydz\\
&+\int_\Omega \big(\KISO(\rhobar)\nabla_3 S\big)\cdot\nabla_3\psi^S\,dxdydz\nonumber
+\int_\Omega \big(\KGM(\rhobar)\nabla_3 S\big)\cdot\nabla_3\psi^S\,dxdydz
=0,\label{STRONG_PE_OCEAN_5}
\end{align}
\end{subequations}
 with $\rho=\rho(\theta,S,p_{st})$ given by (\ref{GIBBS_SPECIFIC_rho}) 
and if $(v,\theta,S)$ satisfies
\begin{equation*}\begin{split}
 &v\in C([0,T],\Lsq)\cap L^2([0,T],\VV),\\
 &\theta\in C([0,T], \HT\cap \LinftyT  )\cap L^2([0,T],\VT),\\
 &S\in C([0,T],\HS\cap \LinftyS)\cap L^2([0,T],\VS),\\
 &\partial_t v\in L^1([0,T],( \Htwo)^*),\\
  &\partial_t \theta\in L^1([0,T], \VT^*),\,\partial_t S\in L^1([0,T], \VS^*).
\end{split}\end{equation*}
$ii)$ The triple $(v,\theta,S)$ is called a strong solution
of (\ref{PE_OCEAN}) on $[0,\mathcal{T}]$ if it 
satisfies (\ref{STRONG_PE_OCEAN}) for all testfunctions $\Phi\in \VV$ and $\phi^\theta\in \VT, \phi^S\in \VS$
and if 
\begin{equation*}\begin{split}
 &v\in C([0,{T}],\VV)\cap L^2([0,{T}], {\Htwo}(\Omega)),\\
 &\theta\in C([0,{T}],\VT)\cap L^2([0,{T}],{\Htwo}(\Omega)),\\
  &S\in C([0,{T}],\VS)\cap L^2([0,{T}],{\Htwo}(\Omega)),\\
 &\partial_t v\in L^1([0,{T}], \Lsq),\\
  &\partial_t (\theta,S)\in L^1([0,{T}], \Lsq).
\end{split}\end{equation*}
\end{definition}
\subsection{Statement of Main Results}\label{SECT_MAIN_RESULT}
\begin{theorem}[Existence of Weak Solutions]\label{THEOREM_EXISTENCE_WEAKSTRONG_SOLUTIONS}
Let $v_0\in \HT$, $\theta_0\in \HT\cap \LinftyT, S_0\in \HS\cap\LinftyS$ 
and the time interval $[0,T], {T}>0$, be given. 
Then there exists a weak solution $(v,\theta,S)$ in the sense of Definition \ref{DEF_STRONGSOLUTION} $i)$
of  (\ref{STRONG_PE_OCEAN}) on $[0,{T}]$. 
\end{theorem}

\begin{corollary}[Small Density Slope Approximation: Existence of Weak Solutions]\label{THEOREM_EXISTENCE_WEAKSTRONG_SOLUTIONS_SMALL_SLOPE}
The assertion of Theorem \ref{THEOREM_EXISTENCE_WEAKSTRONG_SOLUTIONS} remains valid 
if in  (\ref{STRONG_PE_OCEAN}) 
the eddy operators  $\KISO, \KGM$ are replaced by their small density slope approximations $\KISOSMALL, \KGMsmall$, defined in
(\ref{EDDY_OPERATORS_SMALL}).
\end{corollary}
\begin{theorem}[Global Well-Posedness]\label{THEOREM_EXISTENCE_STRONG_SOLUTIONS}
Let $v_0\in \VT$, $\theta_0\in \VT\cap \LinftyT, S_0\in \VS\cap\LinftyS$ 
and the time interval $[0,T], {T}>0$ be given. 
Then there exists a strong solution $(v,\theta,S)$ in the sense of Definition \ref{DEF_STRONGSOLUTION} $ii)$
of (\ref{STRONG_PE_OCEAN}) on $[0,{T}]$. 
This solution is unique and depends continuously on the initial conditions.
\end{theorem}

\begin{corollary}[Global Well-Posedness - Small Density Slope Approximation]\label{THEOREM_EXISTENCE_STRONG_SOLUTIONS_SMALL_SLOPE}
The assertions of Theorem \ref{THEOREM_EXISTENCE_STRONG_SOLUTIONS} remain valid 
if in  (\ref{STRONG_PE_OCEAN}) 
the eddy operators  $\KISO, \KGM$ are replaced by their small density slope approximations $\KISOSMALL, \KGMsmall$, defined in
(\ref{EDDY_OPERATORS_SMALL}).
\end{corollary}

\subsection{Proof of Theorem \ref{THEOREM_EXISTENCE_WEAKSTRONG_SOLUTIONS}}\label{SECT_PROOF_MAIN_RESULT}

\begin{proof}
We use a Galerkin approximation of the system (\ref{STRONG_PE_OCEAN}). 
By ${\bf P}_m$ we denote the projection of the velocity space $\VV$ on $\VV^m:=span\{\psi_k: k=1\ldots m \}$ spanned by eigenfunctions of the Stokes operator. The anisotropic Laplacian operators $R_\theta:=-K_I\triangle - K_D\partial_{zz}^2$ and $R_S:=-K_I\triangle - K_D\partial_{zz}^2$
whose respective domains are the closure of $\VT$ and $\VS$ in ${\Htwo}$, with boundary conditions given by (\ref{PE_AO_BOUNDARY_CONDITIONS_TRAC2}), are positive and self-adjoint, such that each has a compact inverse 
(cf. \cite{CaoTiti_2003}). Consequently there exist orthonormal bases
$(\phi_k^\theta)_k $ and $(\phi_k^S)_k $ of $\Lsq$ of eigenfunctions of $R_\theta$ and $R_S$. Denote $\VT^m:=span\{\phi^\theta_k: k=1\ldots m\}$ and $\VS^m:=span\{\phi^S_k: k=1\ldots m\}$. The projection of $\VT$, $\VS$ on $\VT^m,\VS^m$ is denoted by $P_{\VT^m}, P_{\VS^m}$, respectively. 
\subsubsection*{Step 1: \it Formulation of Approximative System}\label{APPROX_SYSTEM}
We approximate  velocity, temperature and salinity in terms of 
projection on the span of the respective basis functions. With $C_m\in\{\theta_m, S_m\}$ this reads 
\begin{equation}\begin{split}\label{GALERKIN_1}
&v_m(x,y,z,t):={\bf P}_mv(x,y,z,t):=\sum_{k=1}^m a_k(t)\psi_k(x,y,z),\\
&C_m(x,y,z,t):=P_m C(x,y,z,t):=\sum_{k=1}^m b_k^C(t)\phi_k^C(x,y,z).
\end{split}\end{equation}
The Galerkin approximation to (\ref{PE_OCEAN}) reads as follows
\begin{subequations}\label{GALERKIN_2}
\begin{align}
&\partial_t v_m=-{\bf P}_m[(v_m\cdot\nabla) v_m] 
+{\bf P}_m\big[\big(\int_{-h}^z\nabla\cdot v_m(\xi)\, d\xi\big)\partial_z v_m\big] 
-f\vec{k}\times v_m \nonumber\\
&-\nabla \int_{-h}^zg\rho_m(x,y,z',t)\,dz' +\nabla (p_s)_m
+\frac{1}{Re_1}{\bf P}_m\triangle\,v_m +\frac{1}{Re_2}{\bf P}_m\partial^2_{zz} v_m,\label{GALERKIN_4}\\
&\partial_t C_m=- P_m(v_m\cdot\nabla) C_m 
-P_m\big[\big(\int_{-h}^z\nabla\cdot v_m(\xi)\, d\xi\big)\partial_z C_m\big]\nonumber
-P_m\REDIOP(\rhobar_m)(C_m)\nonumber\\
&-P_m\GMOP(\rhobar_m)(C_m),\label{GALERKIN_6}
\end{align}
\end{subequations}
with initial conditions $v_m(t=0):={\bf P}_m v_0, C_m(t=0):=P_m C_0$ 
and with density $\rho_m:=P_m\rho(\theta_m, S_m)$ calculated from $\theta_m,S_m$. 
The Galerkin approximation $\SL_m$ of the isoneutral density slope is defined as
\begin{equation}\begin{split}\label{REDI_OP_SOLVE6}
&\SL_m:=P_m\frac{\nabla_h\rhobar_{m}}{\partial_z\rhobar_{m}}\Pi_{s_0,\epsilon_0}(|\partial_z\rhobar_m|)\pi_{\partial\Omega}.
\end{split}\end{equation} 

\subsubsection*{Step 2: \it Local Existence in Time of Approximative System}\label{LOCAL_EXISTENCE}
The  j'th component of the Galerkin approximation is for $t\in [0,T]$ given by
\begin{equation}\label{GALERKIN_10}
\frac{d}{dt}a_j(t)=G_v^{(j)}(t),\quad\text{and}\quad
\frac{d}{dt}b_j^C(t)=G_C^{(j)}(t) ,
\end{equation}
with right-hand side
\begin{align}
G_v^{(j)}(t)
:=&-\sum_{k,l=1}^m a_k(t)a_l(t)\big<(\psi_k\cdot\nabla)\psi_l, \psi_j\big>_{{\bf L}^2} 
-g\sum_{k}^m d_k\big<\nabla \int_{-h}^z\psi_k(x,y,z')\,dz',\psi_j \big>_{{\bf L}^2} \nonumber \\
& - \sum_{k,l=1}^m a_k(t)a_l(t)\big<\big(\int_{-h}^z\nabla\cdot\psi_l(\xi)\,d\xi\big)\partial_z\psi_k,\psi_j\big>_{{\bf L}^2} \nonumber\\
&-\frac{1}{Re_1}\sum_{k=1}^m a_k(t)\big<\nabla\psi_k,\nabla\psi_j\big>_{{\bf L}^2} 
- \frac{1}{Re_2}\sum_{k=1}^m a_k(t)\big<\partial_z \psi_k,\partial_z\psi_j\big>_{{\bf L}^2},\label{GALERKIN_411}\\
G_C^{(j)}(t)
:=&-\sum_{k=1}^m  a_k(t)b_k^C(t)\big< (\psi_k\cdot\nabla)\phi_k^C,\phi_j^C\big>_{\Lsq}-k_C\sum_{k=1}^m b_k^C(t)\big<\phi_k^C,\phi_j^C\big>_{L^2(\Gamma_u)}\nonumber\\
&-\sum_{k=1}^m a_k(t)b_k^C(t) \big< \big(\int_{-h}^z\nabla\cdot\psi_k(\xi)\,d\xi\big) \partial_z \phi_k^C,\phi_j^C\big>_{\Lsq}\nonumber\\
&-\sum_{k=1}^m b_k(t)\big< (\KISO(\rhobar_k)+\KGM(\rhobar_k))\nabla \phi_k^C,\nabla\phi_j^C\big>_{\Lsq},\label{GALERKIN_412}
\end{align}
where $\big<\cdot,\cdot\big>_\Lsq$ denotes the $L^2$-inner product. 
Equation (\ref{GALERKIN_10}) forms a system of Ordinary Differential Equations that has a unique solution provided the right hand side is locally Lipschitz continuous. The Lipschitz continuity of $G_v^{(j)}$ is classical, as well as the Lipschitz continuity of the advection terms in 
$G_\theta^{(j)}$ and $G_S^{(j)}$. To establish the Lipschitz-property of the eddy matrices $\KISO, \KGM$ it is sufficient to show the Lipschitz continuity of the mapping
\begin{equation}\begin{split}\label{Lipschitz_1}
&(b^\theta,b^S)\mapsto \big< L_\sigma(\rhobar_k(b^\theta, b^S))\partial_\mu \phi_k,\partial_\nu\phi_j\big>_{\Lsq},\\
\end{split}\end{equation} 
where $L_\sigma\in\{\SLx, \SLy\}$, $\partial_\mu,\partial_\nu\in\{\partial_x,\partial_y,\partial_z\}$ and
$b^\theta:=(b_l^\theta(t))_{l=1}^j,b^S:=(b_l^S(t))_{l=1}^j$ are the expansion coefficients of potential temperature and salinity
in terms of their respective basis functions $\phi_k\in\{ \phi_k^\theta,\phi_k^S\}$. 
Let $b^{\theta,(1)}$ 
$b^{S,(1)}$
and $b^{\theta,(2)}$
$b^{S,(2)}$
be two such expansion coefficients and $\rho^{(1)}_k, \rho^{(2)}_k$ the associated densities.
From definition  (\ref{SLOPE_DEF_1}) of $\SL$ we obtain with the continuity of $\Pi_{s_0,\epsilon_0}$, H\"olders inequality
and the convolution properties 
\begin{equation*}\begin{split}
&\big|\big< \big(L_\sigma(\rhobar_k^{(1)}(b^{\theta,(1)}, b^{S,(1)}))-L_\sigma(\rhobar_k^{(2)}(b^{\theta,(2)}, b^{S,(2)})\big)\partial_\mu  \phi_k,
\partial_\nu \phi_j\big>_{\Lsq}\big|\\
&=
\big|\int_\Omega
P_k(\frac{\nabla_h\rhobar_k^{(1)}}{\partial_z\rhobar_k^{(1)}})
\big(\Pi_{s_0,\epsilon_0}(|\partial_z\rhobar_k^{(1)}|)- \Pi_{s_0,\epsilon_0}(|\partial_z\rhobar_k^{(2)}|)\big)\pi_{\partial\Omega}
\partial_\mu \phi_k\partial_\nu \phi_jdxdydz\\
&+\int_\Omega 
P_k\big(\frac{\partial_z\rhobar_k^{(2)}(\nabla_h\rhobar_k^{(1)}-\nabla_h\rhobar_k^{(2)})
+\nabla_h\rhobar_k^{(2)}(\partial_z\rhobar_k^{(2)})-\partial_z\rhobar_k^{(1)})  }{\partial_z\rhobar_k^{(1)}\partial_z\rhobar_k^{(2)}}\big)\times\\
&\times\Pi_{s_0,\epsilon_0}(|\partial_z\rhobar_k^{(2)}|)\pi_{\partial\Omega}
\partial_\mu  \phi_k\partial_\nu\phi_j dxdydz\big|\\
&\leq
c|| \nabla\rhobar_k^{(1)}||_{L^4(\Omega)}
||\partial_\mu\phi_k||_{L^4(\Omega)}  ||\partial_\nu \phi_j||_{L^4(\Omega)}
||\partial_z(\rhobar_k^{(1)}-\rhobar_k^{(2)})||_{L^4(\Omega)} \\
&+c|| \nabla\rhobar_k^{(2)}||_{L^4(\Omega)}
 ||\partial_\mu\phi_k||_{L^4(\Omega)} ||\partial_\nu \phi_j||_{L^4(\Omega)}
  ||\nabla(\rhobar_k^{(1)}-\rhobar_k^{(2)})||_{L^4(\Omega)}\\
\end{split}\end{equation*}
\begin{equation}\begin{split}\label{Lipschitz_2}
&\leq
c(|| \nabla\phi_k^{(1)}||_{L^4(\Omega)}+|| \nabla\phi_k^{(2)}||_{L^4(\Omega)})
||\partial_\mu\phi_k||_{L^4(\Omega)}  ||\partial_\nu \phi_j||_{L^4(\Omega)} ||\nabla_3\phi_k||_{L^4(\Omega)}\times\\
&\times (||b^{\theta,(1)}-b^{\theta,(2)}|| +||b^{S,(1)}-b^{\theta,(2)}||) 
\leq\ell(||b^{\theta,(1)}-b^{\theta,(2)}|| +||b^{S,(1)}-b^{\theta,(2)}||)
\end{split}\end{equation}
where $c,\ell>0$ depends on the $H^2$-norms of  $\phi_j\in H^2(\Omega)$ and on the coefficients of the equation of state and where the regularized density was defined as convolution in (\ref{REG_NEUTRAL_TS}). In the last step 
we have used Lemma \ref{LEMMA_DENSITY_DIFF}. We note that due to the regularity of the expansion functions $\phi_j$ the density regularization
is not required for (\ref{Lipschitz_2}). 

From (\ref{Lipschitz_2}) follows the local Lipschitz continuity of the mapping (\ref{Lipschitz_1}). This implies the (local) Lipschitz continuity of $\KISO,\KGM$. From the Picard Theorem follows that for all $m$ a unique local solution to (\ref{GALERKIN_10}) on time intervals $[0, t_m]$ exists.  
\subsubsection*{Step 3: \it A Priori Bounds on Approximate System}\label{APRIORI_WEAK_SOLUTION}
A priori bounds for velocity, temperature and salinity in $L^\infty([0,T],L^{2})$ and $L^2([0,T],H^1)$
follow analogously to \cite{CaoTiti} with only minor modifications by invoking Lemmas \ref{LEMMA_PROP_REDIOPERATOR} and
\ref{LEMMA_TRACER_BOUNDED}. Steps 3a and 3b below are included for completeness. 
\subsubsection*{Step 3a: \it $L^\infty([0,T],L^{2})$- and $L^2([0,T],H^1)$-bound on tracer}\label{L_BOUND_TRACER_GMR}
 Taking the $L^2$ scalar product of the tracer equation 
(\ref{GALERKIN_6}) with $C_m\in\{\theta_m,S_m\}$ yields after integration-by-parts and with the tracer boundary 
condition (\ref{PE_AO_BOUNDARY_CONDITIONS_TRAC2}),
\begin{equation}\begin{split}\label{PROOF_THM_1_2_GM}
&\frac{1}{2}d_t||C_m||^{2}_{\Lsq}
+
\int_\Omega \big(\KISO(\rhobar_m)\nabla_3 C_m\big)\cdot \nabla_3 C_m\, dxdydz
+k_C||C_m(z=0)||_{L^2(\Gamma_u)}^2=0,
\end{split}\end{equation}
where the  tracer advection term vanishes due to the incompressibility of $v_m$ and the
eddy advection term disappears due to the skewness of $\KGM$. 
This implies with Lemma \ref{LEMMA_PROP_REDIOPERATOR}
and with the Gronwall inequality
\begin{equation}\begin{split}\label{PROOF_THM_1_4_GM}
&||C_m(t)||_{\Lsq}
\leq ||C(t=0)||_{\Lsq}
=:K_1.
\end{split}\end{equation}
From (\ref{PROOF_THM_1_2_GM}) follows after integration with respect to time and with Lemma \ref{LEMMA_PROP_REDIOPERATOR} that
\begin{equation}\begin{split}\label{PROOF_THM_1_5_0_GMc}
&\int_0^t(\int_\Omega 
|\nabla_3 C_m|^2dxdydz)\,ds\leq K_2,
\end{split}\end{equation}
with $K_2(t):=\frac{K_1}{\mu}$, $\mu$ is the constant from Lemma \ref{LEMMA_PROP_REDIOPERATOR}.
\subsubsection*{Step 3b: \it $L^\infty(T,L^{2})$- and $L^2(T,H^1)$-bound on Velocity}
The following estimate can be derived as in \cite{CaoTiti} (see pp. 253-254)
\begin{equation}\begin{split}\label{PROOF_THM_1_13}
&||v_m(t)||_{\Lsq}^2
+\int_0^t \frac{1}{Re_1}||\nabla v_m(s)||_{\Lsq}^2+\frac{1}{Re_2}||\partial_z v_m(s)||_{\Lsq}^2ds\\
&+
||\theta_m(t)||_{\Lsq}^2
+\mu\int_0^t\big[||\nabla_3\theta_m||_{\Lsq}^2+k_\theta||\theta_m(z=0)||_{2}^{2}\big]ds\\
&+
||S_m(t)||_{\Lsq}^2
+\mu\int_0^t||\nabla_3 S_m||_{\Lsq}^2ds\\
&\leq 
||v(t=0)||_{\Lsq}^2
+(gh)^2C_{M}Re_1^2K_1t
+h||v_0||_{\Lsq}^2
=:K_3(t),
\end{split}\end{equation}
where $K_3(t)$ is bounded on $[0,{T}]$. This shows that the approximate solutions $(v_m,\theta_m,S_m)$ exist on the time interval $[0,{T}]$.

\subsubsection*{Step 4: \it Bound on the Time Derivatives in $L^1([0,T], H^{-2}(\Omega))$}\label{LIMIT_TIMEDERIV_WEAK}
For the velocity equation it is well known that $(\partial_t v_m)_m$ is uniformly bounded in $L^{4/3}([0,T], H')$, where $H':=H^{-2}(\Omega)$ is the dual of $\Htwo$ (see e.g. \cite{TEMAM_ZIANE}, sect. 2.3). 
For $\phi\in\Htwo$ it holds
\begin{equation}\begin{split}\label{PROOF_STRONG_SOLUTION_12_WEAK}
&\int_0^t\big<\partial_t C_m(s), \phi\big>ds
\leq
|\int_0^t\big<\REDIOP(\rhobar_m(s))(C_m(s))+\GMOP(\rhobar_m(s))(C_m(s)), \phi\big>\, ds|\\
&+\big|\int_0^t\big(\int_\Omega
\big( (v_m\cdot\nabla) C_m +w_m\partial_zC_m
\big)
\phi\, dxdydz\big)ds\big|.
\end{split}\end{equation}
For the first term on the right-hand side we find with Lemma \ref{LEMMA_PROP_REDIOPERATOR}
\begin{equation}\begin{split}\label{PROOF_STRONG_SOLUTION_14_WEAK}
|\big<\REDIOP(\rhobar_m)(C_m), \phi\big>|\leq M||\nabla C_m||_\Lsq||\nabla \phi||_\Lsq.
 \end{split}\end{equation}
 The second term on the right-hand side is estimated with H\"olders inequality\footnote{For weak solutions the control of the $L^3$-norm of the density slopes $\SL$ requires regularization of the density. In the small slope approximation this is not necessary as it is prescribed that the slopes are bounded.}  
\begin{equation}\begin{split}\label{PROOF_STRONG_SOLUTION_15}
&|\big<\GMOP(\rhobar_m)(C_m), \phi\big>|=\big<\KGM(\rhobar_m)\nabla C_m, \nabla\phi\big>|\\
&\leq c||\SL(\rhobar_m)||_{L^3(\Omega)}    ||\nabla C_m||_{\Lsq}  ||\nabla\phi|| _{L^6(\Omega)}\\
&\leq c||\rho_m||_{\Lsq}||\nabla C_m||_{\Lsq}   ||\phi|| _{\Htwo}\\
&\leq c(||\theta_m||_{\Lsq}+||S_m||_\Lsq)||\nabla C_m||_{\Lsq}   ||\phi|| _{\Htwo},
 \end{split}\end{equation}
For the third term on the right-hand side of (\ref{PROOF_STRONG_SOLUTION_12_WEAK}) it 
follows after integration by parts, with the inequalities of H\"older, Young, and Sobolev's embedding theorem
 \begin{equation}\begin{split}\label{PROOF_STRONG_SOLUTION_16}
&\big|\int_\Omega
\big( (v_m\cdot\nabla_h) C_m +w_m\partial_zC_m
\big)
\phi(x,y,z)\, dxdydz\big|\\
&\leq c ||v_m||_{\Hone}^2||C_m||_{\Lsq}+||C_m||_{\Hone} ||\phi||_{\Htwo}^2.
\end{split}\end{equation}
From (\ref{PROOF_STRONG_SOLUTION_14_WEAK})-(\ref{PROOF_STRONG_SOLUTION_16})
we conclude $(\partial_t C_m)_m$ is bounded uniformly in $L^2([0,T], \Lsq)$.
In particular the Galerkin approximation  $(v_m,\theta_m, S_m)$ is a weak solution of (\ref{PE_OCEAN}).
\subsubsection*{Step 5: \it Passage to the Limit}\label{LIMIT_WEAK_SOLUTION}
The uniform boundedness of $(\partial_t v_m)_m$, $(\partial_t \theta_m)_m$, $(\partial_t S_m)_m$ in
$L^1([0,T], \Lsq)$ and of $(v_m)_m$, $(\theta_m)_m$, $( S_m)_m$ in
$L^2([0,T], \Hone)$ implies with the Lions-Aubin compactness lemma the existence of subsequences 
$(v_n)_n, (\theta_n)_n,( S_n)_n$ such that 
\begin{equation}\begin{split}\label{PROOF_STRONG_SOLUTION_17}
v_n&\to v\qquad\ \ \text{ in}\  L^2([0,T], \Lsq)\text{ strongly},\\
(\theta_n, S_n)&\to (\theta,S)\quad\text{in}\ L^2([0,T], \Lsq)\text{ strongly},\\
(\theta_n, S_n)&\to (\theta,S)\quad\text{in}\ L^2([0,T], \Hone)\text{ weakly}.
\end{split}\end{equation}
These convergence properties allow for all terms in (\ref{STRONG_PE_OCEAN}) to pass to the limit ((see e.g. \cite{TEMAM_ZIANE}, sect. 2.3), except for the eddy terms $\KISO, \KGM$. 

The convergence of $(\theta_n)_n, (S_n)_n$ implies 
that $(\rhoreg_n)_n$ converges to $\rhoreg$ in $L^2([0,T],H^s(\Omega))$ for $s\geq 2$
and that  $\SL_n(\rhoreg_n)$ converges to $\SL(\rhoreg)$ in $L^2([0,T], H^s(\Omega))$ for $s\geq 1$.
For  $\KISO$ the following estimate\footnote{The estimate below that shows convergence of $\REDIOP$ is valid for the regularized density.} holds for $\phi^\theta\in \HtwoT$
\begin{align*}
&\int_\Omega \big(\KISO(\rhobar_n)\nabla_3 \theta_n-\KISO(\rhobar)\nabla_3 \theta\big)\cdot\nabla_3\phi^\theta\, dxdydz
+k_\theta\int_{\Gamma_u}(\theta_n-\theta)\phi^\theta\, dxdy\nonumber\\
=&\int_\Omega \big(\KISO(\rhobar_n)\nabla_3 (\theta_n-\theta)\big)\cdot\nabla_3\phi^\theta\, dxdydz
+k_\theta\int_{\Gamma_u}(\theta_n-\theta)\phi\, dxdy\nonumber\\
&+\int_\Omega  \big(\big(\KISO(\rhobar_n)-\KISO(\rhobar)\big)\nabla_3 \theta \big)\cdot\nabla_3\phi^\theta\, dxdydz\nonumber\\
\leq &
\int_\Omega \big((\KISO(\rhobar_n)\nabla_3 (\theta_n-\theta)\big)\cdot\nabla_3\phi^\theta\, dxdydz\nonumber\\
&+C||\SL(\rhobar_n)-\SL(\rhobar)||_{\Hone}||C_m||_\Hone ||\phi||_\Htwo,\nonumber
\end{align*}
where the right-hand side converges to zero.
For $\KGM$, the convergence of $(\vbolus_n,\wbolus_n)$ to $(\vbolus,\wbolus)$ in $L^2([0,T], \Lsq)$ follows since the density slopes $\SL_n$ converge in $L^2([0,T], \Hone)$ (cf. (\ref{GM_BOLUS})). For salinity analogous results hold.
\end{proof}
\subsection{Proof of Corollary \ref{THEOREM_EXISTENCE_WEAKSTRONG_SOLUTIONS_SMALL_SLOPE}} \label{SECT_PROOF_MAIN_RESULT_SMALL_SLOPE}
({\it Sketch}.) 
For the small density slope approximation only the proof of ellipticity of $\KISO$ in Lemma \ref{LEMMA_PROP_REDIOPERATOR} needs to be modified,
it follows with Young's inequality from
\begin{align*}\label{ELIIPTIC_SMALL_SLOPE}
&\int_\Omega\nabla C\cdot(\KISOSMALL\nabla C)\, dxdydz
=\int_\Omega |\nabla_h C|^2+(\delta+\SL^2)|\partial_z C|^2dxdydz\nonumber\\
&+ \int_\Omega \SLx\partial_zC\partial_xC + \SLy\partial_zC\partial_yC\, dxdydz.
\end{align*}
The rest of the proof of Theorem \ref{THEOREM_EXISTENCE_WEAKSTRONG_SOLUTIONS} 
remains unchanged. 
\subsection{Proof of Theorem \ref{THEOREM_EXISTENCE_STRONG_SOLUTIONS}}\label{SUBSECT_PROOF_EX_STRONG}
\begin{proof}
We infer from Theorem \ref{THEOREM_EXISTENCE_WEAKSTRONG_SOLUTIONS} that a weak solution $(v,\theta,S)$ exists. We proceed by showing that the weak solutions satisfies a priori estimates in $L^\infty_tH^1_x$.

\subsubsection*{Step 1: \it $L^\infty([0,T],\Hone)$-bound on Velocity}\label{STRONG_SOL_VELOC}
The proof of $L^\infty([0,T],\Hone)$-estimates for velocity can be carried out as in Cao-Titi \cite{CaoTiti} (cf. Sections. 3.3.1-3.3.3) and
one arrives at the estimate
\begin{equation}\begin{split}\label{PROOF_STRONG_SOLUTION_1}
&||\nabla_3 v(t)||_{\Lsq}^2
+\int_0^t
\frac{1}{Re_1}||\triangle v(s)||_{\Lsq}^2
+\frac{1}{Re_2}||\nabla\partial_z v(s)||_{\Lsq}^2\, ds\\
&\leq 
||v_0||_{\Hone}e^{\int_0^t K_0(t) \,ds }
+c\int_0^t (||\nabla\theta(s)||_{\Lsq}^2+||\nabla S(s)||_{\Lsq}^2)e^{\int_0^s K_0(\tau)\,d\tau }\,ds,
\end{split}\end{equation}
where the right-hand side 
 is bounded on $[0,{T}]$. 

\subsubsection*{Step 2: \it $L^\infty([0,T],\Hone)$-bound on Temperature and Salinity}\label{STRONG_SOL_TRACER}
Taking the $L^2$ inner product of the tracer equation with $\REDIOP(C)$ yields 
\begin{equation}\begin{split}\label{PROOF_STRONG_SOLUTION_2}
&\int_\Omega \partial_t C\, \REDIOP(\rhobar)(C)\, dxdydz
+||\REDIOP(\rhobar)(C)||_{\Lsq}^2\\
&=
\int_\Omega\big(( v+\vbolus)\cdot\nabla C 
+ (w+\wbolus)\partial_zC\big)
\REDIOP(\rhobar)(C)\, dxdydz.
\end{split}\end{equation}
For the time derivative in (\ref{PROOF_STRONG_SOLUTION_2}) follows with product rule and integration-by-parts
\begin{align}
&\int_\Omega \partial_t C\, \REDIOP(\rhobar)(C)\, dxdydz\nonumber
=\partial_t\int_\Omega  \nabla_3C\cdot (\KISO(\rhobar)\nabla_3C)\, dxdydz\\
&- \int_\Omega \nabla_3C\cdot \big( \partial_t(\KISO(\rhobar))\nabla_3 C \big)\, dxdydz
- \int_\Omega \nabla_3C \cdot  (\KISO(\rhobar)\nabla_3 \partial_t C)\, dxdydz,\label{PROOF_STRONG_SOLUTION_2_time_deriv1}
\end{align}
where $ \partial_t(\KISO(\rhobar))$ denotes the matrix whose entries consist of time derivatives of the entries of $\KISO$. 
For the third integral on the right hand side of (\ref{PROOF_STRONG_SOLUTION_2_time_deriv1}) follows with the boundedness of $\KISO$ 
(Prop. \ref{LEMMA_PROP_REDIOPERATOR}) and Young's inequality
\begin{equation}\begin{split}\label{PROOF_STRONG_SOLUTION_2_time_deriv2}
 &\int_\Omega \nabla_3C\cdot  (\KISO(\rhobar)\nabla_3 \partial_t C)\, dxdydz
\leq (\frac{M}{2\epsilon}||\nabla_3 C||_{\Lsq}^2+\frac{\epsilon}{2}||\nabla_3\partial_t C||_{\Lsq}^2).
\end{split}\end{equation}
The x-component of the second integral on the right-hand side in (\ref{PROOF_STRONG_SOLUTION_2_time_deriv1}) can with 
the inequalities of H\"older, Ladyshenzkaya, Young
be estimated as follows
\begin{equation}\begin{split}\label{PROOF_STRONG_SOLUTION_2_time_deriv4}
&\big|\int_\Omega 
\partial_xC\bigg(\partial_x C 
\partial_t\big (
\frac{1+\delta\SLx^2 +\SLy^2}{1+\SL^2}\big)
+\partial_y C
\partial_t\big (\frac{(\delta-1)\SLx\SLy }{1+\SL^2}\big)
+\partial_z C
\partial_t\big (\frac{(1-\delta)\SLx }{1+\SL^2}\big)\bigg)\, dxdydz\big|\\
&\leq
\int_\Omega 
\big| \partial_xC(\partial_xC+\partial_yC+\partial_zC)\partial_t\SL \big| \, dxdydz
\leq ||\nabla_3 C||_{L^3}||\nabla_3 C||_{L^6}||\partial_t \SL||_{L^2} \\
&\leq \frac{c}{4\epsilon_1\epsilon_2}||\partial_t \SL||_{L^2}^4||\nabla_3 C||_{L^2}^2 +(\frac{\epsilon_1}{2}+\frac{\epsilon_2}{2})|| C||_{H^2}^2
\end{split}\end{equation}
Due to the symmetry of the matrix $\KISO$
analogous estimates apply to y- and z-components of the inner product in the second integral on the right-hand side of (\ref{PROOF_STRONG_SOLUTION_2_time_deriv1}) and
in summary it holds for the time derivative term
\begin{equation}\begin{split}\label{PROOF_STRONG_SOLUTION_2_time_deriv4a}
\big|\int_\Omega \nabla_3C\cdot \big( \partial_t(\KISO(\rhobar))\nabla_3 C \big)\, dxdydz\big|
\leq 
\frac{c}{4\epsilon_1\epsilon_2}||\partial_t \SL||_{L^2}^4||\nabla_3 C||_{L^2}^2 +(\frac{\epsilon_1}{2}+\frac{\epsilon_2}{2})||C||_{H^2}^2.
\end{split}\end{equation}

For the first term on the right hand side of (\ref{PROOF_STRONG_SOLUTION_2})
we obtain by the inequalities of H\"older, Ladyshenzkaya, Young and by Lemma \ref{LEMMA_ELLIPTIC_REG_REDI}
\begin{equation}\begin{split}\label{PROOF_STRONG_SOLUTION_31}
&|\int_\Omega  \big( v\cdot\nabla C\big)  \REDIOP(\rhobar)(C)\,dxdydz |\\
&\leq
c||v||_{L^6(\Omega)}||\nabla C||_{L^3(\Omega)} ||\REDIOP(\rhobar)(C)||_{\Lsq}\\
&\leq
\frac{c}{2\epsilon_3}||v||_{L^6(\Omega)}^2||\nabla C||_{\Lsq}||C||_{\Htwo}+\frac{\epsilon_3}{2} ||\REDIOP(\rhobar)(C)||_{\Lsq}^2\\
&\leq
\frac{c}{4\epsilon_3\epsilon_4}||v||_{\Hone}^4||\nabla_3 C||_{\Lsq}^2
+(\frac{\epsilon_3}{2}+\frac{\epsilon_4}{2})||\REDIOP(\rhobar)(C)||_{\Lsq}^2. 
\end{split}\end{equation}
The term with the eddy-induced horizontal velocity $\vbolus$ in (\ref{PROOF_STRONG_SOLUTION_2}) is estimated analogously, here we use that 
$\vbolus$ vanishes in the vicinity of the boundary such that it suffices to consider the integral over $\Omegainthalf$ 
\begin{equation}\begin{split}\label{PROOF_STRONG_SOLUTION_3}
&|\int_\Omega  \big(\vbolus\cdot\nabla C\big)  \REDIOP(\rhobar)(C)|\,dxdydz |
=|\int_\Omegainthalf  \big(\vbolus\cdot\nabla C\big)  \REDIOP(\rhobar)(C)\,dxdydz |\\\
&\leq
\frac{c}{4\epsilon_5\epsilon_6}||\vbolus||_{\Honeinthalf}^4||\nabla C||_{\Lsq}^2
+(\frac{\epsilon_5}{2}+\frac{\epsilon_6}{2})||\REDIOP(\rhobar)(C)||_{\Lsq}^2\\
&\leq
\frac{c}{4\epsilon_5\epsilon_6}||\SL||_{\Htwointhalf}^4||\nabla_3 C||_{\Lsq}^2
+(\frac{\epsilon_5}{2}+\frac{\epsilon_6}{2})||\REDIOP(\rhobar)(C)||_{\Lsq}^2.
\end{split}\end{equation}
Analogously we find for the integral involving the vertical velocity $\wbolus$  
\begin{equation}\begin{split}\label{PROOF_STRONG_SOLUTION_3vert}
&|\int_\Omega  \big(\wbolus\partial_z C\big) \REDIOP(\rhobar)(C)\,dxdydz |
=|\int_\Omegainthalf  \big(\wbolus\partial_z C\big)  \REDIOP(\rhobar)(C)\,dxdydz |\\\
&\leq
\frac{c}{4\epsilon_7\epsilon_8}||\SL||_{\Htwointhalf}^4||\nabla_3 C||_{\Lsq}^2
+(\frac{\epsilon_7}{2}+\frac{\epsilon_8}{2})||\REDIOP(\rhobar)(C)||_{\Lsq}^2.
\end{split}\end{equation}
The last term  on the right hand side  of (\ref{PROOF_STRONG_SOLUTION_2}) is estimated with 
Prop. 2.2. in \cite{CaoTiti_2003}
 \begin{equation}\begin{split}\label{PROOF_STRONG_SOLUTION_41}
&\int_\Omega\bigg[\big(\int_{-h}^z\nabla\cdot (v(x,y,\xi,t))\, d\xi\big)\partial_zC\bigg]\REDIOP(\rhobar)(C)\, dxdydz\\
&\leq
\frac{c}{4\epsilon_9\epsilon_{10}}||v||_{\Hone}^2||v||_{\Htwo}^2||\partial_z C||_{\Lsq}^2
+(\frac{\epsilon_9}{2}+\frac{\epsilon_{10}}{2})||\REDIOP(\rhobar)(C)||_{\Lsq}^2.
\end{split}\end{equation}
Choosing  in (\ref{PROOF_STRONG_SOLUTION_31}) - (\ref{PROOF_STRONG_SOLUTION_41}) the values
$\epsilon_i=\frac{1}{5}$ yields
\begin{equation}\begin{split}\label{PROOF_STRONG_SOLUTION_6}
&\frac{d}{dt}||\nabla_3 C||^2_{L^2}
+||\REDIOP(\rhobar)(C)||_{\Lsq}^2
\leq
G||\nabla_3C||_{\Lsq}^2,\\
\text{with }&
G:=c
\big(M+||\partial_t \SL||_{L^2} ^4+
||\SL||_{\Htwointhalf}^4
+||v||_{\Hone}^4
+||v||_{\Hone}^2||v||_{\Htwo}^2
\big).
\end{split}\end{equation}
Since $G$ is integrable, 
it follows with the Gronwall inequality that
\begin{equation}\begin{split}\label{PROOF_STRONG_SOLUTION_10}
&||\nabla_3 C(t)||^2_{\Lsq}
\leq
||\nabla_3 C(t=0)||_{\Lsq}^2exp\big(\int_0^t 
G(s)\, ds\big)=:K_C(t),
\end{split}\end{equation}
and it follows $C\in L^\infty([0,{T}],\Hone)$.
Integrating (\ref{PROOF_STRONG_SOLUTION_6})
with respect to time yields with (\ref{PROOF_STRONG_SOLUTION_10}),
\begin{equation}\begin{split}\label{PROOF_STRONG_SOLUTION_11}
&\int_0^t||\REDIOP(C)(s)||_{\Lsq}^2ds
\leq
||\nabla_3 C(t=0)||^2_{\Lsq}+\int_0^t G(s) K_C(s)ds.
\end{split}\end{equation}
This proves that $\REDIOP(C)$ in $L^2([0,T],\Lsq)$. Since $\REDIOP(C)(\cdot,\cdot,\cdot,t)\in \Lsq)$ 
for almost every $t\in [0,{T}]$ it follows 
with Lemma \ref{LEMMA_ELLIPTIC_REG_REDI} that $C\in L^2([0,{T}], {\Htwo})$.
\subsubsection*{Step 3: \it Bound on the Time Derivative in $L^2([0,T], \Lsq)$}\label{LIMIT_TIMEDERIV}
Taking the scalar product of the tracer equation with $\partial_t C_m$ and applying the inequalities of Cauchy-Schwarz and Young to the $\REDIOP$- term yields
\begin{equation}\begin{split}\label{PROOF_STRONG_SOLUTION_12}
&\int_0^{T}||\partial_t C(s)||^2_{\Lsq}ds
\leq
\int_0^{T} ||\REDIOP(\rhobar)(C)(s)||_{\Lsq}^2ds\\
&+2\int_0^{T}\big|\big(\int_\Omega
\big[( v(x,y,z,s)+\vbolus(x,y,z,s))\cdot\nabla C(x,y,z,s)\\
&-\nabla\cdot\big(\int_{-h}^z v(x,y,\xi,s)+\vbolus(x,y,\xi,s)\, d\xi\big)\partial_zC(x,y,z,s)
\big]\partial_t C(x,y,z,s)\, dxdydz\big)\big|ds.
\end{split}\end{equation}
For estimating the right-hand side we proceed as in Step 2 and obtain
\begin{equation}\begin{split}\label{PROOF_STRONG_SOLUTION_14}
&\int_0^{T}|||\partial_t C(s)||^2_{\Lsq}ds
\leq 
\int_0^{T}||\REDIOP(\rhobar)(C)(s)||_{\Lsq}^2
+
H(s)||\nabla_3C(s)||_{\Lsq}^2ds,\\
&\text{with }H:=
||\SL||_{\Htwointhalf}^4
+||v||_{\Hone}^4
+||v||_{\Hone}^2||v||_{\Htwo}^2.
\end{split}\end{equation}
Since $H$ is integrable, it follows with (\ref{PROOF_STRONG_SOLUTION_10}), (\ref{PROOF_STRONG_SOLUTION_11})
that 
$\partial_t C$ is in $L^2([0,{T}], \Lsq)$. Similarly  follows that $\partial_t v\in L^2([0,{T}], \Lsq)$.

\subsubsection*{Step 4: \it Uniqueness and Continuous Dependence on Initial Conditions}\label{STRONG_SOL_UNIQUE}
Consider two strong solutions $(v_1,\theta_1,S_1)$, $(v_2,\theta_2,S_2)$ with initial conditions
$((v_0)_1$,$ (\theta_0)_1$, $(S_0)_1)$, $((v_0)_2,(\theta_0)_2,(S_0)_2)$ 
and denote the differences between them as follows
\begin{equation}\begin{split}
&\hat{v}:=v_1-v_2,\qquad\hat{\theta}:=\theta_1-\theta_2,\qquad \hat{S}:=S_1-S_2, \qquad \hat{\rho}:=\rho_1-\rho_2\\
&
\hat{\vbolus}:=\vbolus_1-\vbolus_2
\hat{C}\in\{\hat{\theta}, \hat{S}\}\qquad
V_i:=v_i+\vbolus_i,\qquad\ \hat{V}=V_1-V_2=\hat{v}+\hat{\vbolus}.
\end{split}\end{equation}
Taking the inner product of the difference equation for velocity with $\hat{v}$ yields
(see \cite{CaoTiti}, p. 262 - 265)
\begin{equation}\begin{split}\label{UNIQUE_VELOC_DIFF_5}
&\frac{d}{d t}|| \hat{v}||_{\Lsq}^2
+||\nabla_3 \hat{v}||_{\Lsq}^2
\leq
F||\hat{v}||_{\Lsq}^2+c||\hat{\rho}||_{\Lsq}^2\\
\text{with }&F:=
\frac{c}{2\epsilon_1} ||\nabla v_2||_{\Lsq}^2+\frac{c}{\epsilon_2}||\partial_z v_2||_{\Lsq}^2||\nabla\partial_z v_2||_{\Lsq}^2
\end{split}\end{equation}
For the difference equation for a generic tracer we obtain similarly
\begin{equation}\begin{split}\label{UNIQUE_TRACER_DIFF_T}
&\frac{d}{d t}|| \hat{C}||_{\Lsq}^2
+\int_\Omega \big(\REDIOP[\rhobar_1](C_1)-\REDIOP[\rhobar_2](C_2)\big)\hat{C}\, dxdydz\\
&\leq
G||\hat{C}||_{\Lsq}^2
+(\frac{\epsilon_4}{2}+\frac{\epsilon_5}{2})||\nabla\hat{C}||_{\Lsq}^2
+(\frac{\epsilon_3}{2}+\frac{\epsilon_6}{2})||\nabla \hat{v}||_{\Lsq}^2\\
&+c(\frac{\epsilon_3}{2}+\frac{\epsilon_6}{2})||\hat{\rhobar}||_{H^3(\Omega)}^2\\
&\text{with }G:=\frac{c}{\epsilon_3\epsilon_4} ||\nabla C_2||_{\Lsq}^4
+\frac{c}{\epsilon_6\epsilon_5}||\partial_z C_2||^2_{\Lsq}||\nabla\partial_z C_2||^2_{\Lsq}
\end{split}\end{equation}
We decompose the second term on the left-hand side as follows
\begin{equation}\begin{split}\label{UNIQUE_DIFF_ALL8a}
&\int_\Omega 
\big(\REDIOP[\rhobar_1](C_1)-\REDIOP[\rhobar_2](C_2)\big) \hat{C}\, dxdydz \\
&=
\int_\Omega 
\big(\widehat{\mathbb{K}}_{iso}\nabla_3 C_1+\KISO[\SL_2]\nabla_3\hat{C}\big)\cdot\nabla_3 \hat{C}\, dxdydz
-k_C\int_{\Gamma_u}|\hat{C}|^2dxdy,
\end{split}\end{equation}
where $\widehat{\mathbb{K}}_{iso}:=\KISO(\rhobar_1)-\KISO(\rhobar_2)$. 
For the second term in (\ref{UNIQUE_DIFF_ALL8a}) the lower bound follows from the ellipticity of $\KISO$
\begin{equation}\begin{split}\label{UNIQUE_DIFF_ALL8b}
\int_\Omega 
\big(\KISO[\SL_2]\nabla_3\hat{C}\big)\cdot\nabla_3 \hat{C}\, dxdydz-k_C\int_{\Gamma_u}|\hat{C}|^2dxdy
\geq \mu||\nabla_3\hat{C}||_{\Lsq}^2,
\end{split}\end{equation}
with $\mu:=\min\{K_I, K_D\}$. For the first term in (\ref{UNIQUE_DIFF_ALL8a}) we have
\begin{equation}\begin{split}\label{UNIQUE_DIFF_ALL8c}
|\int_\Omega 
\big(\widehat{\mathbb{K}}_{iso}\nabla_3 C_1\big)\cdot\nabla_3 \hat{C}\, dxdydz|
&\leq
||\hat{\SL}||_{L^\infty}||\nabla_3 C_1||_\Lsq||\nabla_3\hat{C}||_\Lsq
\end{split}\end{equation}
Collecting (\ref{UNIQUE_DIFF_ALL8a})-(\ref{UNIQUE_DIFF_ALL8c}) and (\ref{UNIQUE_VELOC_DIFF_5})
yields with Lemma \ref{LEMMA_DIFFERENCES}
\begin{equation}\begin{split}\label{UNIQUE_TRACER_DIFF_8d}
&\frac{d}{d t}(|| \hat{v}||_{\Lsq}^2+|| \hat{\theta}||_{\Lsq}^2+|| \hat{S}||_{\Lsq}^2) \\
&\leq
(F+G +c)(|| \hat{v}||_{\Lsq}^2+||\hat{\theta}||_{\Lsq}^2+|| \hat{S}||_{\Lsq}^2) +c||\hat{\SL}||_{L^\infty}\\
&\leq (F+G +c)(|| \hat{v}||_{\Lsq}^2+||\hat{\theta}||_{\Lsq}^2+|| \hat{S}||_{\Lsq}^2).
\end{split}\end{equation}
For strong solutions $F,G$ are integrable and from Gronwall's inequality follows the continuous dependency on the initial conditions and the uniqueness.
\end{proof}
\begin{remark}[Minimal Regularization]\label{remark_minimal_regularization}
The regularization we apply here, has to gain three derivatives (cf. Remark \ref{remark_three_derivatives}). In the context of the $H^1$-regularity investigated here, a potential reduction of the regularization towards two derivatives or less faces two challenges. 
First, to improve the estimates for the nonlinear term  in Step 2  of the proof (see (\ref{PROOF_STRONG_SOLUTION_31})-(\ref{PROOF_STRONG_SOLUTION_41})), which rely on H\"olders inequality and Sobolev embedding. Such an improvement would allow to relax for  the density slope $\SL$ the required regularity. Second, to reduce the regularity requirements for the time derivative of the density gradient (see (\ref{density_eq})). This amounts to improve the regularity estimates of the right-hand side of  (\ref{density_eq}), which comprise, among other terms, the derivative of the transport term. 
\end{remark}

\noindent\subsection*{Acknowledgements}
This work was supported by the Research Unit FOR 5528 ``Mathematical Study of Geophysical Flow Models: Analysis and Computation'' funded by the German Research Foundation DFG. The research of E.S.T. has also benefited from the inspiring environment of the CRC 1114 ``Scaling Cascades in Complex Systems'', Project Number 235221301, Project C06, funded by Deutsche Forschungsgemeinschaft (DFG). \\

\end{document}